\documentclass[11pt]{article}
\usepackage{amsmath,amsthm,amsfonts,amssymb,amscd, amsxtra, mathrsfs,textcomp,verbatim,inputenc}
\usepackage{url}
\usepackage[margin=1.0in]{geometry}
\usepackage{color}
\usepackage{float,mathrsfs,threeparttable,color,soul,colortbl,hhline,rotating,array,multirow,makecell}
\usepackage{tikz}
\usepackage{tkz-euclide}

\usepackage{subcaption}

\definecolor{magenta1}{HTML}{FF00FF}
\definecolor{magenta2}{HTML}{D600AD}
\definecolor{magenta3}{HTML}{B8009B}
\definecolor{magenta4}{HTML}{9A007A}
\definecolor{magenta5}{HTML}{7D0063}
\definecolor{magenta6}{HTML}{64004D}
\definecolor{magenta7}{HTML}{550042}

\newtheorem{theorem}{Theorem}[section]
\newtheorem{lemma}{Lemma}[section]
\newtheorem{definition}{Definition}[section]
\newtheorem{corollary}{Corollary}[section]
\newtheorem{proposition}{Proposition}[section]
\newtheorem{example}{Example}

\newcommand{\R}{\mathbb{R}}
\DeclareMathOperator{\grad}{grad}

\DeclareMathOperator*{\Min}{Minimize}

\newtheorem{algorithm}{Algorithm}

\usepackage{todonotes}
\begin{document} 
\tikzset{
  my tip/.style={
    decoration={
      markings,
      mark=at position 1 with {\arrow[scale=.7]{#1}}
    },
    postaction=decorate
  }
}

\title{Constraint qualifications and strong global convergence properties of an augmented Lagrangian method on Riemannian manifolds}
\author{Roberto Andreani \thanks{Department of Applied Mathematics, University of Campinas, Campinas-SP, Brazil. Email: andreani@ime.unicamp.br}
\and
Kelvin R. Couto\thanks{Federal Institute of Goi\'as, Goi\^ania-GO, Brazil and Department of Applied Mathematics. University of S\~ao Paulo, S\~ao Paulo-SP, Brazil. Email: kelvin.couto@ifg.edu.br}
\and
Orizon P. Ferreira \thanks{Institute of Mathematics and Statistics, Federal University of Goi\'as. Goi\^ania-GO, Brazil. Email: orizon@ufg.br}
\and
Gabriel Haeser \thanks{Department of Applied Mathematics, University of S\~ao Paulo, S\~ao Paulo-SP, Brazil. Email: ghaeser@ime.usp.br}
}
\maketitle

\maketitle
\begin{abstract}
In the past years, augmented Lagrangian methods have been successfully applied to several classes of non-convex optimization problems, inspiring new developments in both theory and practice. In this paper we bring most of these recent developments from nonlinear programming to the context of optimization on Riemannian manifolds, including equality and inequality constraints. Many research have been conducted on optimization problems on manifolds, however only recently the treatment of the constrained case has been considered. In this paper we propose to bridge this gap with respect to the most recent developments in nonlinear programming. In particular, we formulate several well known constraint qualifications from the Euclidean context which are sufficient for guaranteeing global convergence of augmented Lagrangian methods, without requiring boundedness of the set of Lagrange multipliers. Convergence of the dual sequence can also be assured under a weak constraint qualification. The theory presented is based on so-called sequential optimality conditions, which is a powerful tool used in this context. The paper can also be read with the Euclidean context in mind, serving as a review of the most relevant constraint qualifications and global convergence theory of state-of-the-art augmented Lagrangian methods for nonlinear programming.
\end{abstract}

\noindent
{\bf Keywords:} constraint qualifications, global convergence, augmented Lagrangian methods, Riemannian manifolds.

\medskip
\noindent
{\bf AMS subject classification:}  49J52, 49M15, 65H10, 90C30.

\section{Introduction}
The problem of minimizing an objective function defined on a Riemannian manifold has received a lot of attention over the last twenty five  years. Several unconstrained algorithms on Euclidean spaces have been successfully adapted to this more general setting. These adaptations come from the fact that the Riemannian machinery, from a theoretical and practical point of view, allows treating  several  constrained optimization problems as unconstrained Riemannian problems.  It is worth noting that  the works on this subject involve more than merely a theoretical experiment in generalizing Euclidean space concepts to Riemannian manifolds, which is challenging in many different aspects. Unlike Euclidean spaces, Riemannian manifolds are nonlinear objects, making it challenging to develop a solid optimization theory in this setting. The most important thing to keep in mind is that these studies are important mainly because many problems are most effectively addressed from a point of view of Riemannian geometry.  In fact, many optimization problems have an underlying Riemannian geometric structure that can be efficiently exploited with the goal of designing more effective methods to solve them; some references on this subject include  \cite{AbsilBook2008,Boumal2022Book,Edelman1999}.

Although unconstrained Riemannian optimization is already somewhat well established, only a few works have appeared dealing with constrained Riemannian optimization (CRO) problems, that is, Riemannian optimization problems where equality and inequality constraints restrict the variables to a subset of the manifold itself. For instance,  \cite{Yang_Zhang_Song2014} extended to the Riemannian context the Karush/Kuhn-Tucker (KKT) conditions  and  second-order  optimality conditions under a strong assumption, while in \cite{BergmannHerzog2019} a very interesting intrinsic approach was presented for defining suitable KKT conditions. In \cite{Yamakawa_Sato2022} the Approximate-KKT (AKKT) sequential optimality condition was proposed to support the global convergence theory of an augmented Lagrangian method recently introduced in \cite{Liu_Boumal2020}.  In   \cite{Jiang2022} an exact penalty method  for special problems on Stiefel manifolds was presented,  some constraint qualifications and  the first- and second-order optimality conditions  to support the method  are discussed.  A manifold inexact augmented Lagrangian framework to solve a family of nonsmooth optimization problems on Riemannian submanifolds embedded in Euclidean space is proposed in \cite{Deng2022}. In \cite{Obara2022}, a Riemannian sequential quadratic optimization algorithm is proposed, which uses a line-search technique with an $\ell_1$-penalty function as an extension of the standard sequential quadratic optimization algorithm for constrained nonlinear optimization problems in Euclidean spaces. In \cite{IPriemann}, a Riemannian interior point algorithm is introduced.

It is worth mentioning that the theoretical tools needed to support constrained optimization methods on the Riemannian setting  are still under development.  In fact, only recently  in \cite{BergmannHerzog2019} a full theory of constraint qualifications and optimality conditions have been developed, where a definition of weak constraint qualifications (CQs) such as Guignard's CQ and Abadie's CQ have been given. Despite guaranteeing the existence of Lagrange multipliers, more robust applications of these conditions are not known so far, even in the Euclidean setting. This is not the case of stronger conditions such as the linear independence CQ (LICQ) and Mangasarian-Fromovitz CQ (MFCQ), which gives, respectively, uniqueness and compactness of the Lagrange multiplier set, together with boundedness of a typical sequence of approximate Lagrange multipliers generated by several primal-dual algorithms, guaranteeing global convergence to a stationary point. These results were discussed in the Riemannian setting in \cite{Yamakawa_Sato2022}. 

In this paper our goal is to introduce several intrinsic weaker CQs in the Riemannian context, such as the constant rank CQ (CRCQ \cite{Janin_Robert_1984}), the constant positive linear independence CQ (CPLD \cite{Qi_Wei2000}), and their relaxed variants (RCRCQ \cite{Minchenko_Stakhovski_2011} and RCPLD \cite{Andreani_Haeser_Schuverdt_Silva2012RCPLD}). RCPLD is the weakest of these four conditions introduced, however all of them have their own set of applications, which we mention later. With the exception of CRCQ and RCRCQ, which are independent of MFCQ, all CQs presented are strictly weaker than MFCQ. Thus, despite the fact that no such condition guarantees boundedness of the set of Lagrange multipliers at a solution, they are still able to guarantee global convergence of primal-dual algorithms to a stationary point. In particular, we show that all limit points of a safeguarded augmented Lagrangian algorithm will satisfy the KKT conditions under all proposed conditions. Finally, we present two other conditions, the constant rank of the subspace component CQ (CRSC \cite{Andreani_Haeser_Schuverdt_Silva_CRSC_2012}) and the quasinormality CQ (QN \cite{hestenes}), which we also show to be enough for proving the global convergence result we mentioned. Although we do not pursue these results in the Riemannian setting, CRSC is expected to be strictly weaker than RCPLD, while these conditions (CRSC and QN) are the weakest ones known in the Euclidean setting such that an Error Bound condition is satisfied. That is, locally, the distance to the feasible set can be measured by means of the norm of the constraint violation \cite{Andreani_Haeser_Schuverdt_Silva_CRSC_2012,eb2}, which should be an interesting result to be extended to the Riemannian setting. Throughout the text we review several results known in the Euclidean setting in order to serve as a guide for future extensions to Riemannian manifolds. We chose to present in the Riemannian setting an interesting characteristic of QN which is related to the global convergence of the augmented Lagrangian method; namely, under QN, the sequence of approximate Lagrange multipliers generated by the algorithm is bounded, guaranteeing primal-dual convergence even when the set of Lagrange multipliers is itself unbounded, a result that first appeared in \cite{gnep} in the Euclidean setting. In order to do this, we will need to define a stronger sequential optimality condition known as Positive-AKKT (PAKKT \cite{Andreani_Fazzio_Schuverdt_Secchin2019}). Finally, the machinery of sequential optimality conditions we introduce is relevant due to the fact that it is easy to extend the global convergence results we present to other algorithms. We want also to draw attention to the fact that all of the findings obtained in this study are also valid in Euclidean spaces, thus, this study may also be seen as a review of the recent developments in constraint qualifications and their connections with global convergence of algorithms in the Euclidean setting.

This paper is organized as follows.  Section~\ref{sec:aux} presents some definitions and  preliminary results   that are important throughout our study.  In Section~\ref{sec:prel}, we state  the CRO  problem  and  also recall  the KKT  and  AKKT conditions, together with the definitions of LICQ and MFCQ for CRO problems.  Section~\ref{sec:nStricCQ}   is devoted to introducing the new  CQs for the CRO problem, namely,  (R)CRCQ, (R)CPLD, and CRSC, where we  present several examples and the proof that these conditions are indeed CQs associated with the global convergence of the augmented Lagrangian method. In Section~\ref{eq:PAKKT}  we introduce the PAKKT condition and the quasinormality CQ, where we show that the Lagrange multiplier sequences generated by the augmented Lagrangian method is bounded under quasinormality. The last section contains some concluding remarks.
\subsection{Notations, terminology and basics results} \label{sec:aux}

In this section, we recall some notations and basic concepts of Riemannian manifolds used throughout the paper. They can be found in many books on Riemannian Geometry, see, for example, \cite{Lang1995,Sakai1996,Loring2011}.

Let  ${\cal M}$ be an $n$-dimensional smooth Riemannian manifold. Denote  the {\it tangent space} at a point $p$  by $T_p{\cal M}$,   the {\it tangent bundle}  by $T{\cal M}:= \bigcup_{p\in M}T_p{\cal M}$  and a {\it vector field} by a mapping   $X\colon {\cal M} \to T{\cal M}$ such that $X(p) \in T_p{\cal M}$. Assume also that ${\cal M}$ has a {\it Riemannian metric} denoted by  $\langle  \cdot,   \cdot \rangle$ and the corresponding {\it norm}  by $\|\cdot\|$.  For $f\colon U \to\mathbb{R}$ a   differentiable function with derivative $d f(\cdot)$, where $U$ is an open subset of the manifold ${\cal M}$,   the Riemannian metric induces the mapping   $f\mapsto  \grad f $  which  associates     its {\it gradient vector field} via the following  rule  $\langle \grad f(p),X(p)\rangle:=d f(p)X(p)$, for all $p\in U$.  The {\it length} of a piecewise smooth curve  $\gamma\colon[a,b]\rightarrow {\cal M}$ joining $p$ to $q$ in ${\cal M}$, i.e., $\gamma(a)= p$ and $\gamma(b)=q$ is denoted by  $\ell(\gamma)$.  The {\it Riemannian distance} between $p$ and $q$ is defined as
$
d(p,q) = \inf_{\gamma \in \Gamma_{p,q}} \ell(\gamma),
$
where $\Gamma_{p,q}$ is the set of all piecewise smooth curves in ${\cal M}$ joining points $p$ and $q$. This distance induces the original topology on ${\cal M}$, namely $({\cal M}, d)$ is a complete metric space and the bounded and closed subsets are compact. The {\it open} and {\it closed balls} of radius $r>0$, centered at $p$, are respectively defined by $B_{r}(p):=\left\{ q\in {\cal M}:~ d(p,q)<r\right\}$ and $ B_{r}[p]:=\left\{ q\in {\cal M} :~d(p,q)\leq r\right\}$.  Let $\gamma$ be a curve joining the points $p$ and $q$ in ${\cal M}$ and let $\nabla$ be the Levi-Civita connection associated to $({\cal M}, \langle \cdot, \cdot \rangle)$.  A vector field $Y$ along a smooth curve $\gamma$ in ${\cal M}$ is said to be {\it parallel}  when $\nabla_{\gamma^{\prime}} Y=0$. If $\gamma^{\prime}$ itself is parallel, we say that $\gamma$ is a {\it geodesic}.  A Riemannian manifold is {\it complete} if its geodesics $\gamma(t)$ are defined for any value of $t\in \mathbb{R}$.  From now on, {\it ${\cal M}$ denotes  an $n$-dimensional smooth and complete Riemannian manifold}. Owing to the completeness of the Riemannian manifold ${\cal M}$, the {\it exponential map} at $p$, $\exp_{p}\colon T_{p}{\cal M} \to {\cal M}$, can be given by $\exp_{p}v = \gamma(1)$, where $\gamma$ is the geodesic defined by its position $p$ and velocity $v$ at $p$ and $\gamma(t) = \exp_p(tv)$ for any value of $t$.  For $p\in {\cal M}$, the {\it injectivity radius} of ${\cal M}$ at $p$ is defined by
$$
 r_{p}:=\sup\{ r>0:~{\exp_{p}}{\vert_{B_{r}(0_{p})}} \mbox{ is a diffeomorphism} \},
 $$
where $B_{r}(0_{p}):=\lbrace  v\in T_{p}{\cal M}:~\|v\| <r\rbrace$ and  $0_{p}$ denotes the origin of $T_{p}{\cal M}$.  Hence, for $0<\delta<r_{p}$ and  $\exp_{p}(B_{\delta}(0_{ p})) = B_{\delta}({p})$,  the map  $\exp^{-1}_{p}\colon  B_{\delta}({p})  \to B_{\delta}(0_{ p})$  is a diffeomorphism. Moreover,  for all $p, q\in B_{\delta}({p})$, there exists a unique geodesic segment $\gamma$ joining  $p$ to $q$, which is given by $\gamma_{p q}(t)=\exp_{p}(t \exp^{-1}_{p} {q})$, for all $t\in [0, 1]$. Furthermore,  $d(q,p)\,=\|exp^{-1}_{p}q\|$ and the map  $ B_{\delta}({p})\ni q\mapsto \frac{1}{2}d({q}, p)^2$ is  $C^{\infty}$   and its gradient is given by 
\[
\grad \frac{1}{2}d({q}, p)^2=-\exp^{-1}_{p}{q},
\]
 see, for example, \cite[Proposition 4.8, p.108]{Sakai1996}.

Next we state some elementary facts on (positive-)linear dependence/independence of gradient vector fields, whose proofs are straightforward.
Let $h=(h_1,\dots,h_s)\colon {\cal M} \to {\mathbb R}^s$ and  $g=(g_1,\dots,g_m)\colon {\cal M} \to {\mathbb R}^m$ be continuously differentiable functions  on ${\cal M}$. Let us denote
\begin{equation} \label{eq:ap}
A(q, {\cal I}, {\cal J}):=\{\grad h_i(q) :~i\in{\cal I} \}\cup\{\grad g_j(q):~j \in {\cal J}\}, \qquad q\in {\cal M}, 
\end{equation} 
where ${\cal I} \subset \{1, \ldots, s\}$,  ${\cal J}\subset \{1, \ldots, m\}$ while $\{\grad h_i(q) :~i\in{\cal I} \}\cup\{\grad g_j(q):~j \in {\cal J}\}$ is a multiset, that is, repetition of the same element is allowed.

 \begin{definition}\label{d:positive-linearly-dependent}
 Let $V = \{v_1, \dots , v_s\}$ and $W = \{w_1, \dots , w_m\}$ be two finite multisets on $T_{p}{\cal M}$. The pair $(V,W)$ is said to be positive-linearly dependent if there exist $\alpha=(\alpha_1, \ldots, \alpha_s) \in \R^{s}$ and $\beta=(\beta_1, \ldots, \beta_m)\in \R^{m}_+$ such that $(\alpha, \beta)\neq 0$ and
\begin{equation*}
\sum_{i=1}^s\alpha_iv_i+ \sum_{j=1}^m\beta_jw_j=0.
\end{equation*}
Otherwise, $(V,W)$ is said to be positive-linearly independent. When clear from the context, we refer to $V\cup W$ instead of $(V,W)$.
\end{definition}

\begin{lemma} \label{lemma:LD}
Let $p\in {\cal M}$ and  assume that $A(p,{\cal I}, {\cal J} )$ is (positive-)linearly independent.  Then, there exists   $\epsilon>0$ such that $A(q,{\cal I}, {\cal J}) $ is also (positive-)linearly independent for all $q \in B_{\epsilon}(p)$. 
\end{lemma} 

\begin{lemma}\label{cr:eqcrcq}
The following two conditions are equivalent:
\begin{enumerate}
\item[(i)] There exists $\epsilon >0$ such that  for all  ${\cal I}\subset \{1, \ldots, s\}$ and ${\cal J}\subset \{1, \ldots, m\}$,  whenever  the set  $A(p, {\cal I}, {\cal J})$ is linearly dependent, $A(q, {\cal I}, {\cal J})$ is also linearly dependent for all $q\in B_{\epsilon}(p)$.
\item[(ii)] There exists $\epsilon >0$ such that   for all ${\cal I}\subset \{1, \ldots, s\}$ and ${\cal J}\subset \{1, \ldots, m\}$ the rank of $A(q, {\cal I}, {\cal J})$ is constant for any $q\in B_{\epsilon}(p)$.
\end{enumerate}
\end{lemma}

\begin{lemma}[Carath\'eodory's Lemma \cite{Andreani_Haeser_Schuverdt_Silva2012RCPLD}]\label{l:Caratheodory}
Let   $u_1, \ldots , u_s, v_1, \ldots , v_m$ be vectors in a  finite-dimensional vector space $V$ such that $\{u_1,\dots,u_s\}$ is linearly independent. Suppose $x \in V$ is such that there are real scalars $ \alpha_1, \ldots , \alpha_s, \beta_1, \ldots , \beta_m$, with $\beta_j \neq 0 $ for $j=1, \ldots , m$ and
\begin{equation*}
x= \sum_{i=1}^s\alpha_i u_i + \sum_{j=1}^m \beta_j v_j.
\end{equation*}
Then, there exist a subset ${\cal J} \subset \{1, \ldots , m\}$, and real scalars $ \bar{\alpha}_i, i=1,\dots,s$, and $\bar{\beta}_j\neq0, j\in{\cal J}$ such that
\begin{equation*}
x= \sum_{i=1}^m\bar{\alpha}_i u_i + \sum_{j \in {\cal J}} \bar{\beta}_j v_j,
\end{equation*}
$\bar{\beta}_j\beta_j>0$ for all $j\in{\cal J}$, and $\{u_i:i \in \{1,\dots,s\}\}\cup \{v_j :j\in {\cal J}\}$  is linearly independent.
\end{lemma}

We end this section by  stating  some  standard notations in Euclidean spaces.  The set of all $m \times n$ matrices with real entries is denoted by $\R^{m \times n}$ and    $\R^m\equiv \R^{m\times 1}$.  For $M \in \R^{m\times n} $ the matrix $ M^{\top}  \in \R^{n\times m}$ is the  {\it transpose} of $M$. For all $x, y \in \mathbb{R}^m$, 
$\min\{x,y\}\in\R^m$ is the component-wise minimum of $x$ and $y$. We denote by $[y]_+$ the Euclidean projection of $y$ onto the non-negative orthant $\R^m_+$, while $\left\|y\right\|_2$ and $\left\| y \right\|_{\infty}$ denote its Euclidean and infinity norms, respectively.

\section{Preliminaries} \label{sec:prel}
In this paper we are interested in the following Constrained Riemannian Optimization (CRO) problem
\begin{equation}\label{PNL}
\begin{array}{l}
\displaystyle\Min_{q\in {\cal M}}f(q),\\
\mbox{subject~to~}h(q)=0, ~ g(q)\leq 0,
\end{array}
\end{equation}
where ${\cal M}$ is an $n$-dimensional smooth and complete Riemannian manifold, the functions  $f\colon {\cal M} \to {\mathbb R}$,  $h=(h_1,\dots,h_s)\colon {\cal M} \to {\mathbb R}^s$ and  $g=(g_1,\dots,g_m)\colon {\cal M} \to {\mathbb R}^m$
are continuously differentiable on ${\cal M}$. 
The feasible set $\Omega\subset{\cal M}$ of problem~\eqref{PNL} is defined by
\begin{equation} \label{eq:constset}
\Omega:=\{q\in {\cal M}:~h(q)=0, ~ g(q)\leq 0\},
\end{equation}
which is closed. For a given point $p\in \Omega$, let  ${\cal A}(p)$ be the set of indexes of active inequality constraints, that is,
\begin{equation} \label{eq:actset}
{\cal A}(q):=\left\{j\in \{1, \ldots, m\}:~g_j(q)= 0\right\}.
\end{equation}
We say that the Karush/Kuhn-Tucker (KKT) conditions are satisfied at $p\in\Omega$ when there exists so-called Lagrange multipliers $(\lambda, \mu) \in  {\mathbb R}^s\times {\mathbb R}^m_{+}$ such that the following two conditions hold:
 \begin{itemize} 
\item[(i)]$\grad L(p, \lambda, \mu)=0$,
\item[(ii)]$\mu_j=0$, for all $j\notin {\cal A}(p)$, 
\end{itemize}
where $L(\cdot, \lambda, \mu): {\cal M}\to \mathbb{R}$ is the {\it Lagrangian function} defined by
\begin{equation*} 
L(q, \lambda, \mu):=f(q)+\sum_{i=1}^s\lambda_ih_i(q)+ \sum_{j=1}^m\mu_jg_j(q).
\end{equation*}
For $p\in\Omega$, the {\it linearized/linearization cone} ${\cal L} (p)$ is defined as
\begin{equation*}
{\cal L} (p) := \big\{ v \in T_p{\cal M}:~ \left\langle \grad h_i(p) , v\right\rangle =0,~i=1, \ldots s;~ \left\langle \grad g_j(p) , v \right\rangle \leq 0, ~ j \in {\cal A}(p)   \big\},
\end{equation*}
and its polar is given by
\begin{equation}\label{eq:PolarConeLin}
{\cal L} (p)^{\circ}= \Big\{v \in  T_p{\cal M} :~ v= \sum_{i=1}^{s} \lambda_i \grad h_i (p) + \sum_{j \in {\cal A}(p)} \mu_j \grad g_j(p), ~\mu_j \geq 0,  \lambda_i\in {\mathbb R}\Big\}.
\end{equation}
A constraint qualification (CQ) is a condition that refers to the analytic description of the feasible set and that guarantees that every local minimizer is also a KKT point. In \cite{BergmannHerzog2019} it was shown that when $p$ is a local minimizer of \eqref{PNL} that satisfies Guignard's CQ, that is, ${\cal L} (p)^{\circ}={\cal T} (p)^{\circ}$, where ${\cal T} (p)$ is the Bouligand tangent cone of $\Omega$ at $p$, then the KKT conditions are satisfied at $p$. 


In~\cite{Udriste_1988} the convex inequality constrained problem is studied, under a Slater CQ, on a complete Riemannian manifold. In this case, the objective and inequality constraints are convex along geodesics and the feasible set is described by a finite collection of inequality constraints. In this context KKT conditions are formulated. In \cite{Yang_Zhang_Song2014} it was shown that when $p$ is a local solution of \eqref{PNL} and LICQ holds at $p$, that is, the set
$
\{\grad h_i(p):~i=1, \ldots , s \}\cup \{\grad g_j(p):~ j\in {\cal A}(p)\}
$
is linearly independent, then the KKT conditions are satisfied at $p$.

Without CQs, an approximate verison of the KKT conditions are known to be satisfied at local minimizers:

\begin{theorem} \label{def:AKKT}
Let $p\in\Omega$ be a local minimizer of \eqref{PNL}. Then $p$ is an {\it Approximate-KKT} (AKKT) point, that is, there exist sequences $(p^k)_{k\in {\mathbb N}}\subset {\cal M}$, $(\lambda^k)_{k\in {\mathbb N}}\subset {\mathbb R}^s$ and $(\mu^k)_{k\in {\mathbb N}}\subset {\mathbb R}_+^m$ such that 
\begin{enumerate}
\item[(i)] $\lim_{k\to \infty}p^k=p,$
\item[(ii)] $ \lim_{k\to \infty} \grad L(p^k, \lambda^k, \mu^k)=0,$
\item[(iii)] $ \mu^k_j=0$,  for all $j\notin {\cal A}(p)$ and sufficiently large $k$.
\end{enumerate}
\end{theorem}

This result appeared in \cite{Yamakawa_Sato2022}, as an extension of the well known nonlinear programming version of this theorem \cite{AndreaniHaeserMatinez2011}. Any sequence $(p^k)_{k\in {\mathbb N}}\subset {\cal M}$ that satisfies $(i)$, $(ii)$ and $(iii)$ above is called a primal AKKT sequence for $p$ while the correspondent sequence $(\lambda^k,\mu^k)_{k\in\mathbb{N}}$ is its dual sequence. In the Euclidean setting, this notion has shown to be crucial in developing new constraint qualifications and expanding global convergence results of several algorithms in different contexts; for instance, nonlinear programming \cite{Andreani_Haeser_Schuverdt_Silva2012RCPLD,ccp}, Nash equilibrium problems \cite{gnep}, quasi-equilibrium problems \cite{qep}, multi-objective \cite{giorgi}, second-order cone programming \cite{seq-crcq-socp}, semidefinite programming \cite{AHV,seq-crcq-sdp,weaksparsecq}, Banach spaces \cite{Borgens}, equilibrium constraints \cite{mpecakkt}, cardinality constraints \cite{kanzowcard}, among several other applications and extensions. In \cite{Yamakawa_Sato2022} the following safeguarded augmented Lagrangian algorithm was defined and it was proved that its iterates are precisely AKKT sequences. In order to define it, we denote by $\mathcal{L}_\rho(\cdot,\lambda,\mu):{\cal M}\to\mathbb{R}$ the standard Powell-Hestenes-Rockafellar augmented Lagrangian function, defined by $$\mathcal{L}_\rho(q,\lambda,\mu):=f(q)+\frac{\rho}{2}\left(\left\|h(q)+\frac{\lambda}{\rho}\right\|^2+\left\|\left[g(q)+\frac{\mu}{\rho}\right]_+\right\|^2\right).$$

\begin{algorithm} {\bf  Safeguarded augmented Lagrangian algorithm} \label{Alg:LAA}
\begin{description}
	\item[ Step 0.] Take  $p^0 \in {\cal M}$,   $ \tau \in \left[0  \, , \, 1\right) $, $\gamma >1$, $ \lambda_{\min} < \lambda_{\max}$, $\mu_{\max} >0$, and $\rho_1>0$. Take also  $\bar{\lambda}^1 \in \left[\lambda_{\min}, \lambda_{\max}\right]^s$ and $\bar{\mu}^1 \in \left[0, \lambda_{\max}\right]^m$  initial Lagrange multipliers estimates, and  $\left(\epsilon_k\right)_{k\in {\mathbb N}}\ \subset  \R_{+}$  a sequence of tolerance parameters such that $\lim_{k\rightarrow \infty} \epsilon_k =0$.  Set  $k \leftarrow 1$.
	
	\item[ Step 1.] (Solve the subproblem) Compute (if possible) $p^k$ such that
\begin{equation}\label{step1 alg1}
\left\|\grad {\mathcal{L}}_{\rho_k} (p^k, \bar{\lambda}^k, \bar{\mu}^k) \right\| \leq \epsilon_k.
\end{equation}
If it is not possible, stop the execution of the algorithm, declaring failure.

	\item[ Step 2.] (Estimate new multipliers) Compute
\begin{equation}\label{step2 alg1}
\lambda^{k}=\bar{\lambda}^k + \rho_kh(p^k), \qquad \qquad  \mu^{k}=\left[ \bar{\mu}^k + \rho_k g(p^k)\right]_{+}.
\end{equation}

	\item[ Step 3.] (Update the penalty parameter) Define
\begin{equation}\label{eq:step3_Vk}
V^{k}=\frac{\mu^{k}-\bar{\mu}^k}{\rho_k}.
\end{equation}
If $k=1$ or
\begin{equation}\label{eq: step3-2 alg}
 \max\left\{\big\|h(p^{k})\big\|_2 \, , \, \big\|V^k\big\|_2\right\} \leq \tau   \max\left\{\big\|h(p^{k-1}) \big\|_2 \, , \, \big\|V^{k-1}\big\|_2\right\},
\end{equation}
choose $\rho_{k+1} = \rho_k$. Otherwise, define $\rho_{k+1}=\gamma \rho_k$.
	
	\item[ Step 4.]  (Update multipliers estimates) Compute $ \bar{\lambda}^{k+1} \in \left[ \lambda_{\min} \, , \, \lambda_{\max}\right]^m$ and  $ \bar{\mu}^{k+1} \in \left[ 0 \, , \, \mu_{\max}\right]^p$.
	
	\item[ Step 5.] (Begin  a new iteration) Set $k \leftarrow k+1$ and go to {\bf Step 1}.
\end{description}
\end{algorithm}

In the algorithm, $(\lambda^k,\mu^k)_{k\in {\mathbb N}}$ is the dual sequence associated with $(x^k)_{k\in {\mathbb N}}$, which may be unbounded, while the safeguarded dual sequence $(\bar{\lambda}^k,\bar{\mu}^k)_{k\in {\mathbb N}}$ is bounded and used for defining the subproblems. A standard choice is considering $(\bar{\lambda}^{k+1},\bar{\mu}^{k+1})$ as the projection of $(\lambda^k,\mu^k)$ onto the corresponding box. It was shown in \cite{Yamakawa_Sato2022} that any limit point of a sequence $(p^k)_{k\in\mathbb{N}}$ generated by Algorithm~\ref{Alg:LAA} is such that it is stationary for the problem of minimizing an infeasibility measure, namely
\begin{equation*}
\displaystyle \Min_{q\in {\cal M}} \frac{1}{2}\left\|h(q)\right\|_2^2 + \frac{1}{2}\left\|g(q)_{+}\right\|_2^2.
\end{equation*}
When the limit point is feasible, they showed that it is an AKKT point. The correspondent AKKT sequence is precisely the sequence $(p^k)_{k\in\mathbb{N}}$ of primal iterates, what can be attested by the dual sequences $(\lambda^k,\mu^k)_{k\in\mathbb{N}}$ generated by the algorithm.

Thus, in order to establish a standard global convergence result to Algorithm~\ref{Alg:LAA}, namely, by showing that its feasible limit points satisfy the KKT conditions, it is sufficient to consider any condition that guarantees that all AKKT points are in fact KKT points. Due to Theorem \ref{def:AKKT}, the said condition will necessarily be a CQ. Constraint qualifications with this additional propery are sometimes called {\it strict} CQs, and only the following ones have been stated in the Riemannian context:

\begin{definition} 
 Let $\Omega$  be  given by \eqref{eq:constset}, $p\in \Omega$ and  ${\cal A}(p)$ be given by~\eqref{eq:actset}.  
 The point $p$ is said to satisfy:
\begin{enumerate}
 \item[(i)]  the linear independence constraint qualification (LICQ) if 
$$
\{\grad h_i(p):~i=1, \ldots , s \}\cup \{\grad g_j(p):~ j\in {\cal A}(p)\}
$$
is linearly independent;
\item[(ii)] the Mangasarian-Fromovitz constraint qualification (MFCQ) if 
$$
\{\grad h_i(p):~i=1, \ldots , s \}\cup\{\grad g_j(p):~ j\in {\cal A}(p)\}
$$
is positive-linearly independent.\end{enumerate}
\end{definition}
The definition of LICQ was presented in \cite{Yang_Zhang_Song2014} while MFCQ was introduced in~\cite{BergmannHerzog2019}, where it was shown that LICQ implies MFCQ. In the next section we will introduce several new weaker CQs and we will prove that they can still be used for proving global convergence to a KKT point of algorithms that generate AKKT sequences such as Algorithm~\ref{Alg:LAA}.

\section{New strict constraint qualifications} \label{sec:nStricCQ}

We will say that a property ${\cal P}$ of the constraints defining the feasible set $\Omega$ of \eqref{PNL} at a given point $p\in\Omega$ is a {\it strict} CQ for the necessary optimality condition ${\cal S}$ if at a point $p\in\Omega$ that satisfies both ${\cal P}$ and ${\cal S}$, it is necessarily the case that $p$ satisfies the KKT conditions, according to the definition given in~\cite{Birgin_Martinez_2014_SCQ}. Thus, after we present and discuss our conditions, we shall prove that they are all strict CQs with respect to the sequential optimality condition AKKT from Theorem \ref{def:AKKT}. The first conditions we propose are the following:
\begin{definition} \label{def:NewCQS}
Let $\Omega$  be  given by \eqref{eq:constset}, $p\in \Omega$ and  ${\cal A}(p)$ be given by~\eqref{eq:actset}. The point  $p$ is said  to satisfy:
\begin{enumerate}
\item[(i)] the  constant rank constraint qualification (CRCQ) if there exists $\epsilon>0$ such that for all  ${\cal I}\subset \{1, \ldots, s\}$ and ${\cal J}\subset {\cal A}(p)$ the rank of $\{\grad h_i(q):~i\in {\cal I} \}\cup \{\grad g_j(q) : j\in {\cal J}\}$ is constant for all $q\in B_{\epsilon}(p)$; 
\item[(ii)] the  constant positive linear dependence condition (CPLD), if  for any  ${\cal I}\subset \{1, \ldots, s\}$ and ${\cal J}\subset {\cal A}(p)$,  whenever  the set 
$\{\grad h_i(p):~i\in {\cal I} \}\cup \{\grad g_j(p) : j\in {\cal J}\}$  is positive-linearly dependent, there exists  $\epsilon >0$ such that  $\{\grad h_i(q):~i\in {\cal I} \}\cup \{\grad g_j(q):~ j\in {\cal J}\}$ is linearly dependent,  for all $q\in B_{\epsilon}(p)$;
\item[(iii)] the  Relaxed-CRCQ (RCRCQ)  if there exists $\epsilon >0$ such that   for all ${\cal J}\subset {\cal A}(p)$ the rank of $\{\grad h_i(q):~i=1, \ldots , s \}\cup \{\grad g_j(q):~ j\in {\cal J}\}$ is constant for all $q\in B_{\epsilon}(p)$;
\item[(iv)] the Relaxed-CPLD (RCPLD) if there exists $\epsilon >0$ such that the following two conditions hold:
\begin{enumerate}
      \item[(a)] the rank of $\{\grad h_i(q):~i=1, \ldots , s\}$ is constant for all $q\in B_{\epsilon}(p)$;
      \item[(b)] Let  ${\cal K}  \subset \{1, \ldots , s\}$ be such that $\{\grad h_i(p):~ i \in {\cal K}  \}$ is  a basis for the subspace generated by $\{\grad h_i(p):~ i=1, \ldots, s \}$. For all ${\cal J} \subset {\cal A}(p)$, if $\{\grad h_i(p):~i\in {\cal K}\}\cup \{\grad g_j(p):~ j\in {\cal J}\}$ is   positive-linearly dependent, then $\{\grad h_i(q):~i\in {\cal K}\}\cup \{\grad g_j(q):~ j\in {\cal J}\}$ is   linearly dependent, for all $q\in B_{\epsilon}(p)$.
\end{enumerate}
\end{enumerate}
\end{definition}
These conditions are natural versions in the Riemannian setting of the existing conditions in the Euclidean setting. The following diagram shows the relationship among the conditions introduced so far, where an arrow represents strict implication, which shall be proved later in this section.

\begin{figure}[!htb] 
	\centering
\begin{tikzpicture}
		\tikzset{
		    old/.style={rectangle, very thick, minimum size=0.7cm,fill=black!30,draw=black},
		    new/.style={rectangle,very thick, minimum size=0.7cm,fill=black!5,draw=black},
		    v/.style={-stealth,very thick,black!80!black},
		    u/.style={-stealth,dashed,very thick,blue!80!black},
		}
		\node[old] (LICQ) at (0,0) {LICQ};
		\node[new] (CRCQ) at (3.5,0) {CRCQ};
		\node[new] (RCRCQ) at (7,0) {RCRCQ};	
		\node[old] (MFCQ) at (0,-2) {MFCQ};
		\node[new] (CPLD) at (3.5,-2){CPLD};
		\node[new] (RCPLD) at (7,-2) {RCPLD};
		\draw[v] (LICQ) -- (CRCQ);
		\draw[v] (LICQ) -- (MFCQ);
		\draw[v] (MFCQ) -- (CPLD);
		\draw[v] (CRCQ) -- (RCRCQ);
		\draw[v] (CRCQ) -- (CPLD);		
		\draw[v] (RCRCQ) --(RCPLD);
		\draw[v] (CPLD) -- (RCPLD);

	\end{tikzpicture}
	\caption{Strict constraint qualifications for problem~\eqref{PNL}.}
	\label{relations}
\end{figure}
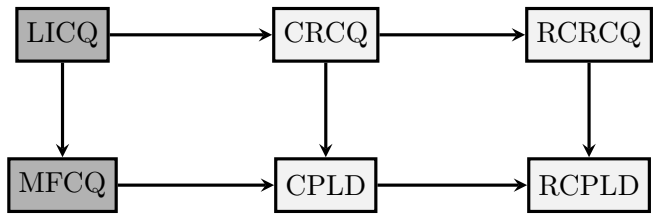

The reason for introducing these CQs in the Riemannian setting is due to their several applications known in the Euclidean case, which we expect to be extended also to the Riemannian setting. Although we shall only prove results concerning the global convergence of the safeguardaded augmented Lagrangian method, let us briefly review some properties of these conditions in the Euclidean setting.

LICQ is equivalent to the uniqueness of the Lagrange multiplier for any objective function that assumes a constrained minimum at the point \cite{wachsmuth}. However, it is considered to be too stringent. For instance, it fails when the same constraint is repeated twice in the problem formulation. On the other hand, MFCQ and its many equivalent statements is the most prevalent CQ in the nonlinear programming literature, with several applications. In particular, it considers the correct sign of the Lagrange multiplier in its formulation, what can be though as it being a more adequate statement than LICQ. However, the simple trick of replacing an inequality constraint $h(x)=0$ by two inequalities $h(x)\leq0$ and $-h(x)\leq0$ is enough for ensuring that MFCQ does not hold. This is due to the fact that under this very natural formulation, the set of Lagrange multipliers (if non-empty) is necessarily unbounded, while MFCQ is equivalent to the boundedness of this set \cite{wachsmuth}. Notice also that the case of linear constraints is not automatically covered by any of these two assumptions, which generally requires a separate analysis when one is assuming MFCQ or LICQ.

Condition CRCQ, on the other hand, gives more freedom to someone modeling an optimization problem, given that it is not tricked by repetition of a constraint or by splitting an equality constraint in two inequalities. It also subsumes the linear case, disregarding a separate analysis. However, it does not consider the correct sign of Lagrange multipliers, being thus independent of MFCQ. The CPLD condition, on the other hand, corrects this issue, introducing the correct sign considering positive linear dependence instead of standard linear dependence (see the alternative definition of CRCQ as given by Lemma \ref{cr:eqcrcq}), being then strictly weaker than both MFCQ and CRCQ together. This condition has been used mainly for showing global convergence of algorithms, firstly for an SQP method, when it was introduced in \cite{Qi_Wei2000}, and it was popularized for being the basis for the global convergence theory of the popular ALGENCAN software \cite{Birgin_Martinez_2014_SCQ}. However, other applications have emerged such as in bilevel optimization \cite{bilevel1,bilevel2}, switching constraints \cite{switching}, exact penalty \cite{exactp}, among several others. On the other hand, CRCQ is more robust in terms of applications, since it has been introduced in order to compute the derivative of the value function \cite{Janin_Robert_1984}. It has also found applications in the characterization of tilt stable minimizers \cite{gfrerer}. More interestingly, while MFCQ is still able to provide a second-order necessary optimality condition, the condition depends on the whole set of Lagrange multipliers, which limits its practical use. CRCQ, on the other hand, provides a strong second-order necessary optimality condition depending on a single Lagrange multiplier \cite{Andreani_Haeser_Schuverdt_Silva2012RCPLD}, a result which was recently extended to the context of conic optimization \cite{CRCQfaces}. This difference with respect to MFCQ in terms of the second-order results can somehow be explained by the fact that under CRCQ the value of the quadratic form defined by the Hessian of the Lagrangian evaluated at a direction in the critical cone is invariant to the Lagrange multiplier used to define it \cite{gfrerer}.

The relaxed versions of CRCQ \cite{Minchenko_Stakhovski_2011} and CPLD \cite{Andreani_Haeser_Schuverdt_Silva2012RCPLD} can in fact be thought as the ``correct'' versions of these conditions, as they enjoy the very same properties previously described. In fact, there is no reason to consider all subsets of equality constraints, and this was already present in the very first results proved under CRCQ by Janin in \cite{Janin_Robert_1984}.

Let us now prove that the strict implications shown in Figure \ref{relations} hold for any Riemannian manifold ${\cal M}$ with dimension $n\geq2$. We do this by providing several examples that help illustrate the different conditions we propose, however we only describe in details the computations concerning the most relevant examples; the other ones being analogous. We start by showing in the next two examples that MFCQ and CRCQ are independent conditions, that is, in Example \ref{ex:CRNotMF}, CRCQ holds and MFCQ fails while in Example \ref{ex:MFNotCR}, MFCQ holds while CRCQ fails.

\begin{example} \label{ex:CRNotMF}
Consider problem~\eqref{PNL} with feasible set $\Omega:=\{q\in {\cal M}:~ h(q)\leq0, -h(q)\leq0\}$, where $h: {\cal M} \to {\mathbb R}$ is continuously differentiable on ${\cal M}$ with $\grad h(p)\neq0$ and $p\in\Omega$. Thus MFCQ fails at $p$ while CRCQ holds.
\end{example}

To proceed with the examples let us define some auxiliary functions.  Let ${\cal M}$ be an $n$-dimensional Riemannian manifold and  $p\in {\cal M}$.   Take  $0<{\bar \delta}<r_p$, where $r_p$ is the injectivity radius, such that   $B_{{\bar \delta}}(p)$ is a strongly convex  neighborhood, which exists by~\cite[Proposition~4.2]{doCarmo1992}. Let  $u, v\in T_{p}{\cal M}$ with $\|u\|=\|v\|=1$, $\langle v, u\rangle =0$, and define  the geodesics  $\gamma_u(t):=\exp_p(tu)$ and $\gamma_v(t):=\exp_p(tv)$.  Take also   $0<\delta< {\bar \delta}$ and define   $p_1:= \gamma_u(-\delta)$, $p_2:= \gamma_u(\delta)$,  and $p_3= \gamma_v(\delta)$.  Note that  $p_1, p_2, p_3 \in B_{{\bar \delta}}(p)$ with $p_1 \neq p_2$, $p_1\neq p_3$, and  $p_2\neq p_3$. Define the following auxiliary functions 
\begin{equation} 
\label{defphi}
\varphi_i(q)= \frac{1}{2} d(q, p_i)^2-\frac{1}{2}d(p, p_i)^2, \qquad i=1,2,3.
\end{equation}

\begin{example} \label{ex:MFNotCR}
 Define the  functions $g_1(q):=\varphi_1(q)$, $g_2(q):=-\varphi_2(q)$ and consider a feasible set $\Omega:=\{q\in {\cal M}:~  g(q)\leq 0\}$, where $g:= (g_1, g_2)$. One can see that CRCQ is not valid at $p\in\Omega$ while MFCQ holds.
\end{example}

In~\cite{Qi_Wei2000}, it was proved that CPLD is strictly weaker than both MFCQ and CRCQ together in the Euclidean setting. The next example shows that the same thing happens for any smooth Riemannian manifold ${\cal M}$ with dimension $n\geq2$. 

\begin{example} \label{ex:CPLDNotCRCQ_MFCQ}
Let ${\cal M}$ be an $n$-dimensional  Riemannian manifold   with $n\geq  2$. Take  $p\in {\cal M}$ and   $g:=(g_1, \ldots , g_4)\colon{\cal M}\to {\mathbb R}^4$  continuously differentiable functions satisfying the following conditions
\begin{enumerate}
\item[(i)]  $g(p)=0$;
\item[(ii)] $\grad g_1(p)= \grad g_2(p)\neq0$ and  $\grad g_3(p)= - \grad g_4(p)\neq0$;
\item[(iii)] for all $\epsilon >0$, there exists  $q \in B_{\epsilon} (p)$ such that $\{\grad g_1({q}), \grad g_2({q})\}$ is linearly independent with  $q \neq p$;
\item[(iv)] $\{\grad g_1(p), \grad g_3(p)\}$ is linearly independent.
\item[(v)] there exists $\epsilon>0$ such that $\{\grad g_3(q), \grad g_4(q)\}$ is linearly dependent, for all $q \in B_{\epsilon} (p)$.
\end{enumerate}
Consider the feasible set $\Omega:=\{q\in {\cal M}:~  g(q)\leq 0\}$. Then, by $(i)$, $p \in \Omega$. It  follows from  the first equality in condition $(ii)$, that the set $\{\grad g_1(p), \grad g_2(p)\}$ is linearly dependent. Thus, $(iii)$ implies that $p$ does not satisfy CRCQ.  Furthermore,  the second  equality in condition $(ii)$ guarantees that $p$ does not satisfy MFCQ. We will now show that $p$ satisfies CPLD. It is easy to see that the set $\{\grad g_j (p):~j \in {\cal J}\subset {\cal A}(p) \}$ is positive linearly dependent if, and only if $\{3,4\}\subset{\cal J}$. Therefore, by $(v)$ we concluded that $p$ satisfies CPLD.  In the following we build two concrete examples satisfying conditions $(i),(ii), (iii), (iv)$, and $(v)$. The first one considers ${\cal M}$ as the $2$-sphere while the second one considers an arbitrary $2$-dimensional manifold ${\cal M}$, however it is easy to generalize the examples to an arbitrary dimension $n\geq2$.
\begin{itemize} 
\item Consider  the sphere  ${\cal M}:=\{(x, y, z)\in {\mathbb R}^3:~x^2+y^2+z^2=1\}$ and take  $p:=(0,0,1)$. The functions   $g_1(x, y, z):=x$, $g_2(x, y, z):=x+y^2$, $g_3(x, y, z):=x+y$, $g_4(x, y, z):=-x-y$  satisfy conditions $(i)$, $(ii)$, $(iii)$, $(iv)$, and $(v)$. Indeed,  for $q=(x,y,z)\in B_\epsilon(p)$ we have $\grad g_1(q)=\Pi_{T_q{\cal M}}(1,0,0)$, $\grad g_2(q)=\Pi_{T_q{\cal M}}(1,2y,0)$, $\grad g_3(q)=\Pi_{T_q{\cal M}}(1,1,0)$, and $\grad g_4(q)=\Pi_{T_q{\cal M}}(-1,-1,0)$, where $\Pi_{T_q{\cal M}}$ denotes the orthogonal projection onto $T_q{\cal M}$. Cleary, $(i)$ holds. To see that $(ii)$ and $(iv)$ hold it is enough to note that since $T_p{\cal M}=\{p\}^{\perp}$, at $q=p$ the projections coincide with the vectors being projected.   We proceed to prove that $(iii)$ holds. Let $q:=(x,y,z)$ with $y\neq0$ and $z\neq0$, hence $u:=(1,0,0)$, $v:=(1,2y,0)$, and $q$ are linearly independent. Take $\alpha,\beta\in\R$ such that $\alpha\grad g_1(q)+\beta\grad g_2(q)=0$. Since  $T_q{\cal M}=\{q\}^{\perp}$, this implies that $\alpha(u-r_uq)+\beta(v-r_vq)=0$ for some $r_u,r_v\in\R$, which in turn gives $\alpha u+\beta v+(-\alpha r_u-\beta r_v)q=0$, implying $\alpha=\beta=0$, hence, $(iii)$. We obtain $(v)$ by noting that $\grad g_3(q)=-\grad g_4(q)$ for all $q$.

\item Let ${\cal M}$ be a $2$-dimensional complete manifold. Let us show that the functions $g_1(q):=\varphi_1(q)$, $g_2(q):=-\varphi_2(q)$, $g_3(q):=\varphi_3(q)$, and  $g_4(q):=-\varphi_3(q)$ satisfy conditions $(i)$, $(ii)$, $(iii)$, $(iv)$, and $(v)$, where these functions are defined in \eqref{defphi}. Indeed,  $g(p)=0$, which gives $(i)$. Since $\grad g_1(p)=-\exp^{-1}_pp_1=\delta u\neq0$, $\grad g_2(p)=\exp^{-1}_pp_2=\delta u$, $\grad g_3(p)=-\exp^{-1}_pp_3=-\delta v\neq0$ and $\grad g_4(p)=\exp^{-1}_pp_3=\delta v$, then $g$  satisfies $(ii)$.

We proceed to show that  $g_1$ and  $g_2$ satisfy  (iii). For that, take a point  $q\in B_{{\bar \delta}}(p)$ with   $q\neq p$ such that  $d({q},p_2)=d({q},p_1) < d(p_2,p_1)$.   Note that  $\grad g_1({q})=-\exp^{-1}_{q}p_1$ and $\grad g_2({q})=\exp^{-1}_{q}p_2$. In addition, due to $d({q},p_2)=d({q},p_1)$, we have $\|\grad g_1({q})\|=\|\grad g_2({q})\|=d({q},p_1)$.  Assume by contradiction that $\{\grad g_1({q}), \grad g_2({q})\}$ is linearly dependent. Since ${\cal M}$ is  $2$-dimensional and $d({q},p_1) < d(p_2,p_1) $, we conclude that $\grad g_1({q})=\grad g_2({q})$.  Consider the geodesic   
$$
\gamma(t)=\exp_{q}(-t\exp^{-1}_{q}p_1)= \exp_{q}(t\exp^{-1}_{q}p_2), 
$$
where the second equality holds because  we are under the assumption $\grad g_1({q})=\grad g_2({q})$. Hence $\gamma(0)=q$,  $\gamma(-1)=p_1$, and $\gamma(1)=p_2$. Considering that there exists a unique geodesic joining $p_1$ and $p_2$, we conclude that $\gamma_u=\gamma$. Thus, there exists ${\bar t}$ such that  $\gamma({\bar t})=p$ and 
$$
\gamma'({\bar t})=P_{q\gamma({\bar t})}(-\exp^{-1}_{q}p_1).
$$
We know that the parallel transport is an isometry and $\gamma'({\bar t})=\delta \gamma'_u(0)=-\exp^{-1}_pp_1$. Thus, using the last equality we conclude that $d(q, p_1)=d(p, p_1)$. Hence, considering that $q$, $p$ and $p_1$ belongs to the same geodesic,  we have  $q=p$, which is a contradiction. Therefore, $\{\grad g_1({q}), \grad g_2({q})\}$ is linearly independent  for   $q \neq p$.

Due to $\grad g_1(p)=-\delta u$, $\grad g_3(p)=\delta v$ and $\langle v, u\rangle =0$, condition $(iv)$ is satisfied. Finally,  due to $g_4=-g_3$ we have  $\grad g_4({q})=-\grad g_3({q})$ and    $(v)$ is also satisfied.

The situation in consideration is depicted in Figure~\ref{figura CPLD NotCRCQ_MFCQ}.
\end{itemize}

\end{example}

\begin{figure}[!htb]
\begin{subfigure}[b]{0.45\linewidth}
\centering
\begin{tikzpicture}[scale=0.4, transform shape]
\tkzInit[xmin=-10,xmax=10,ymin=-6,ymax=6];
\tkzDefPoint(-7,-1){A}
\tkzDefPoint(7,1){B}
\tkzDefPoint(3,-5){C}
\fill [magenta!10!white] (-7,-1) to [out=50, in=140] (7,1) to [out=188, in=75] (3,-5) to [out=110, in=-7] (-7,-1);
\draw (-7,-1) to [out=50, in=140] (7,1) to [out=188, in=75] (3,-5) to [out=110, in=-7] (-7,-1);
\tkzDefPoint(-2.35,2.35){D}
\tkzDefPoint(-5.6,0){E}
\tkzDefPoint(4,-0.99){F}
\tkzDefPoint(6,1.6){G}
\tkzDrawSegment(D,E)
\tkzDrawSegment(E,F)
\tkzDrawSegment(F,G)
\tkzDrawSegment(G,D)
\tkzDrawPolygon[fill=cyan!10](D,E,F,G);
\tkzDefPoint(-0.06,0.73){O}
\tkzDefPoint(4,3.5){I}
\tkzDefPoint(4,0.5){J}
\tkzDefPoint(4.1,0.687){L}
\tkzDefPoint(-6,-0.7){M}
\tkzLabelPoint[left, magenta!100!red](M){${\cal M}$}
\tkzDefPoint(-4.6,-0.1){S}
\tkzLabelPoint[above, cyan!100!blue, font=\fontsize{8}{5}\selectfont](S){$T_p{\cal M}$}
\tkzDefPoint(1.6,0.1){Q1}
\tkzDefPoint(-1.95,0.7){Q2}
\tkzDefPoint(0.2,1.7){Q3}
\node[right, magenta1] at (Q1) {$p_2$};
\node[left, magenta2] at (Q2) {$p_1$};	
%
\node[right, magenta3] at (Q3) {$p_3$};	
\draw[thick,magenta2, opacity=0.3, smooth, line width=0.5pt] [rotate=40] plot[domain=-161:168,variable=\t] ({1.8*cos(\t)-1}, {sin(\t)*1.9+1.8},{sin(\t)});
\draw[thick,magenta2, opacity=1, smooth, line width=0.5pt] [rotate=40] plot[domain=168:254,variable=\t] ({1.8*cos(\t)-1}, {sin(\t)*1.9+1.8},{sin(\t)});
\draw[thick,magenta1, opacity=0.3, smooth,  line width=0.5pt] [rotate=-20] plot[domain=-21:228.5,variable=\t] ({1.8*cos(\t)+1.5},{ sin(\t)*1.9+0.7},{sin(\t)});
\draw[thick,magenta1, opacity=1, line width=0.5pt] [rotate=-20] plot[domain=228.5:343,variable=\t] ({1.8*cos(\t)+1.5},{ sin(\t)*1.9+0.7},{sin(\t)});
\draw[thick, magenta3, rotate=20, smooth, opacity=1, line width=0.5pt] plot[domain=-9.3:122.5, variable=\t] ({1.3*cos(\t)+0.9},{ 1.5*sin(\t)+1.7},{sin(\t)+0.4});	
\draw[thick, magenta3, rotate=20, smooth, opacity=0.3, line width=0.5pt] plot[domain=122.5:355.7, variable=\t] ({1.3*cos(\t)+0.9},{ 1.5*sin(\t)+1.7},{sin(\t)+0.4});		
\draw[magenta!50!white] plot [smooth] coordinates {(Q1) (O) (Q2)}; 
\draw[magenta!50!white] plot [smooth] coordinates {(Q3) (0.12,1.5) (O) }; 
\draw[my tip=stealth, magenta1] (-0.06,0.73) -- (1.8,0.6) node [above right, magenta1] {$\grad g_1 (p) = \grad g_2 (p)$};
\draw[my tip=stealth, magenta3] (-0.06,0.73) -- (0.13,1.75) node[left , magenta3] {$\grad g_3 (p) $};
\draw[my tip=stealth, magenta3] (-0.06,0.73) -- ($(0.13,1.75)!2!(-0.06,0.73)$)node[right, magenta3] {$\grad g_4 (p) $};
\tkzLabelPoint[below left](O){$p$}
\tkzDrawPoint(O)
\node[circle, magenta1, fill, inner sep=0.8pt] (Q1) at (1.6,0.1){};
\node[circle, magenta2, fill, inner sep=0.8pt] (Q2) at (-1.95,0.7){};
\node[circle, magenta3, fill, inner sep=0.8pt] (Q3) at (0.2,1.7){};
\end{tikzpicture}
\caption{ The scenario at point $p$.}
\end{subfigure}
\hspace{0.5cm}
\begin{subfigure}[b]{0.45\linewidth}
\centering
\begin{tikzpicture}[scale=0.4, transform shape]
\tkzInit[xmin=-10,xmax=10,ymin=-6,ymax=6];
\tkzDefPoint(-7,-1){A}
\tkzDefPoint(7,1){B}
\tkzDefPoint(3,-5){C}
\fill [magenta!10!white] (-7,-1) to [out=50, in=140] (7,1) to [out=188, in=75] (3,-5) to [out=110, in=-7] (-7,-1);
\draw (-7,-1) to [out=50, in=140] (7,1) to [out=188, in=75] (3,-5) to [out=110, in=-7] (-7,-1);
\tkzDefPoint(-2.35,2.7){D}
\tkzDefPoint(-5.6,-0.55){E}
\tkzDefPoint(3.7,-1.55){F}
\tkzDefPoint(6,2.05){G}
\tkzDrawSegment(D,E)
\tkzDrawSegment(E,F)
\tkzDrawSegment(F,G)
\tkzDrawSegment(G,D)
\tkzDrawPolygon[fill=cyan!17](D,E,F,G); 
\tkzDefPoint(-0.06,0.73){O}
\tkzDefPoint(-0.09,0.53){Q}
\tkzDefPoint(4,3.5){I}
\tkzDefPoint(4,0.5){J}
\tkzDefPoint(4.1,0.687){L}
\tkzDefPoint(-6,-0.7){M}
\tkzLabelPoint[left, magenta!100!red](M){${\cal M}$}
\tkzDefPoint(-4.6,-0.6){S}
\tkzLabelPoint[above, cyan!100!blue, font=\fontsize{8}{5}\selectfont](S){$T_q{\cal M}$}
\node[circle, magenta1, fill, inner sep=0.7pt] (Q1) at (1.6,0.1){};
\node[below left, magenta1] at (Q1) {$p_2$};
\node[circle, magenta2, fill, inner sep=0.7pt] (Q2) at (-1.95,0.7){};
\node[left, magenta2] at (Q2) {$p_1$};	
\node[circle, magenta3, fill, inner sep=0.7pt] (Q3) at (0.25,1.75) {};
\node[below right, magenta3] at (Q3) {$p_3$};	
\draw[thick,magenta2, opacity=0.3, smooth, line width=0.3pt] [rotate=40] plot[domain=0:360,variable=\t] ({1.8*cos(\t)-1}, {sin(\t)*1.9+1.8},{sin(\t)});
\draw[thick,magenta4, opacity=0.3, smooth, dash pattern=on 0.2pt off 0.3pt, line width=0.3pt] [rotate=40] plot[domain=0:360,variable=\t] ({1.85*cos(\t)-1}, {sin(\t)*1.95+1.8},{sin(\t)});
\draw[thick,magenta1, opacity=0.3, smooth,  line width=0.3pt] [rotate=-20] plot[domain=-49.4:260.1,variable=\t] ({1.8*cos(\t)+1.5},{ sin(\t)*1.9+0.7},{sin(\t)});
\draw[thick,magenta1, opacity=1, line width=0.3pt] [rotate=-20] plot[domain=260.1:310.6,variable=\t] ({1.8*cos(\t)+1.5},{ sin(\t)*1.9+0.7},{sin(\t)});
\draw[thick,magenta4, opacity=0.3, smooth, dash pattern=on 0.2pt off 0.3pt,  line width=0.3pt] [rotate=-20] plot[domain=-46:257,variable=\t] ({1.85*cos(\t)+1.5},{ sin(\t)*1.95+0.7},{sin(\t)});
\draw[thick,magenta4, opacity=1, dash pattern=on 0.2pt off 0.3pt, line width=0.3pt] [rotate=-20] plot[domain=257:314,variable=\t] ({1.85*cos(\t)+1.5},{ sin(\t)*1.95+0.7},{sin(\t)});	
\draw[thick, magenta3, rotate=20, opacity=0.3, smooth, line width=0.3pt] plot[domain=96.1:379.5, variable=\t] ({1.3*cos(\t)+0.9},{ 1.5*sin(\t)+1.7},{sin(\t)+0.4});
\draw[thick, magenta3, rotate=20, opacity=1, smooth, line width=0.3pt] plot[domain=19.5:96.1, variable=\t] ({1.3*cos(\t)+0.9},{ 1.5*sin(\t)+1.7},{sin(\t)+0.4});	
\draw[thick, magenta3, rotate=20, opacity=0.3, smooth, dash pattern=on 0.2pt off 0.3pt, line width=0.3pt] plot[domain=108.5:370, variable=\t] ({1.5*cos(\t)+0.9},{ 1.7*sin(\t)+1.7},{sin(\t)+0.4});	
\draw[thick, magenta3, rotate=20, opacity=1, smooth, dash pattern=on 0.2pt off 0.3pt, line width=0.3pt] plot[domain=10:108.5, variable=\t] ({1.5*cos(\t)+0.9},{ 1.7*sin(\t)+1.7},{sin(\t)+0.4});	
\draw[my tip=stealth, magenta2] (-0.09,0.53) -- (1.8,0.7) node [right, magenta2] {$\grad g_2 (q)$};
\draw[my tip=stealth, magenta1] (-0.09,0.53) -- (1.8,0.3) node [right, magenta1] {$\grad g_1 (q) $};
\draw[my tip=stealth, magenta3] (-0.09,0.53) --  (0.15,1.8) node[left, magenta3] {$\grad g_3 (q) $};
\draw[my tip=stealth, magenta3] (-0.09,0.53) -- ($(0.15,1.8)!2!(-0.09,0.53)$)  node[left, magenta3] {$\grad g_4 (q) $};
\tkzLabelPoint[above left](O){$p$}
\tkzDrawPoint(O)
\tkzLabelPoint[below left](Q){${q}$}
\tkzDrawPoint(Q)

\end{tikzpicture}
\caption{Case for $q$ in the neighborhood of $p$.}
\end{subfigure}

	\caption{Illustrative figure for Example~\ref{ex:CPLDNotCRCQ_MFCQ} where CPLD holds while MFCQ and CRCQ fails. The rank of the gradients indexed by $\{1,2\}$ is not constant, but they are positive-linearly independent. In addition, the gradients in $\{3,4\}$ are positive-linearly dependent, which implies that MFCQ fails, but this remains to be the case in a neighborhood.}
\label{figura CPLD NotCRCQ_MFCQ} 
\end{figure}
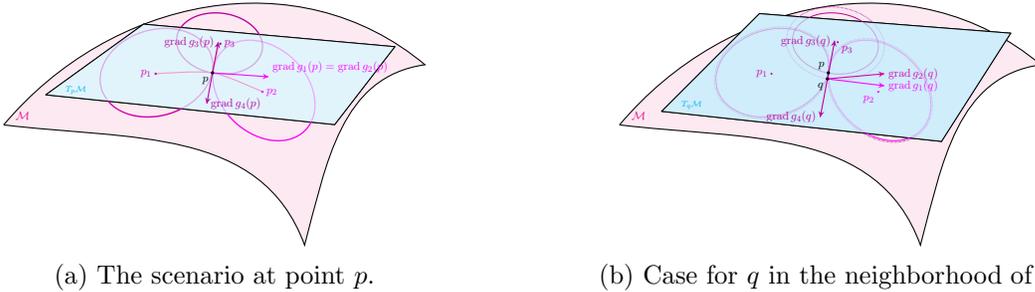


As proved in~\cite{Minchenko_Stakhovski_2011}, RCRCQ is strictly weaker than CRCQ in the Euclidean context. This fact is also true for any smooth Riemannian manifold ${\cal M}$ with dimension $n\geq2$. In fact, the following example shows that RCRCQ does not imply CPLD.

\begin{example} \label{ex:CPLDNotRCRCQ}
  Define the functions $h(q):=-\varphi_3(q)$, $g_1(q):=-\varphi_1(q)$,  $g_2(q):=-\varphi_2(q)$ from \eqref{defphi} and consider a feasible set $\Omega:=\{q\in {\cal M}:~ h(q)=0,  g(q)\leq 0\}$, where  $g:= (g_1, g_2)$. The point   $p \in \Omega$  does not  satisfy CPLD, but it satisfies RCRCQ.
\end{example}

Finally, let us show that RCRCQ implies RCPLD. We do this by providing the following equivalent description of RCRCQ:


\begin{proposition}\label{P:RCRCQ}
 Let  ${\cal K}  \subset \{1, \ldots , s \}$ be such that $\{\grad h_i(p):~ i \in  {\cal K}\}$  is a basis for the subspace generated by $\{\grad h_i(p):~ i=1, \ldots, s \}$.  RCRCQ  holds at $p\in \Omega$ if and only if there exists $\epsilon >0$ such that the following two conditions hold:
\begin{enumerate}
 \item[(i)] the rank of $\{\grad h_i(q):~i=1, \ldots , s\}$ is constant for all $q\in B_{\epsilon}(p)$;
  \item[(ii)] for all ${\cal J} \subset {\cal A}(p)$, if $\{\grad h_i(p):~i\in {\cal K}\}\cup \{\grad g_j(p):~ j\in {\cal J}\}$ is   linearly dependent, then $\{\grad h_i(q):~i\in {\cal K}\}\cup \{\grad g_j(q):~ j\in {\cal J}\}$ is  linearly dependent for all $q\in B_{\epsilon}(p)$. 
\end{enumerate} 
\end{proposition}
\begin{proof}

Assume first that $p \in \Omega$ satisfies RCRCQ. Taking ${\cal J} = \emptyset$ in the definition of RCRCQ, we obtain $(i)$. In order to obtain $(ii)$, let ${\cal J} \subset {\cal A}(p)$ such that $\{\grad h_i(p):~i\in {\cal K}\}\cup \{\grad g_j(p):~ j\in {\cal J}\}$ is linearly dependent. Since $\{\grad h_i(p):~ i \in  {\cal K}\}$  is a basis for the subspace generated by $\{\grad h_i(p):~ i=1, \ldots, s \}$ and $(i)$, we have that there exists $\epsilon>0$ such that $\{\grad h_i(q):~ i \in  {\cal K}\}$  is a basis for the subspace generated by $\{\grad h_i(q):~ i=1, \ldots, s \}$ for all $q\in B_{\epsilon}(p)$. Thus, in accordance with RCRCQ, the rank of $\{\grad h_i(q):~i\in {\cal K}\}\cup \{\grad g_j(q):~ j\in {\cal J}\}$ is constant for all $q\in B_{\epsilon}(p)$. Consequently, $(ii)$ holds.

To prove the reciprocal assertion, let ${\cal J} \subset {\cal A}(p)$. It is worth noting that, owing to Lemma ~\ref{lemma:LD}, the rank of the set cannot decrease in a neighborhood. Let us choose ${\cal J}_1 \subset {\cal J}$ such that $A(p, {\cal K}, {\cal J}_1)$ is a basis for $A(p, {\cal K}, {\cal J})$ -- see the notation introduced in \eqref{eq:ap}. The case  ${\cal J}_1 = {\cal J}$ follows trivially. Consider the situation ${\cal J}_1 \neq {\cal J}$ and let $j \in {\cal J} $ with $j \notin {\cal J}_1 $. As a result of $(ii)$, $A(q, {\cal K}, {\cal J}_1) \cup \{\grad g_j(q)\}$ must continue to be linearly dependent for $q$ in a neighborhood of $p$. Therefore, rank  $A(q, {\cal K}, {\cal J}) = \left|{\cal K}\right| + \left|{\cal J}_1\right| $ for all $q\in B_{\epsilon}(p)$ and sufficiently small $\epsilon>0$. Considering the definition of ${\cal K}$ and $(i)$ we must have that $A(q, \{1,\dots,s\}, {\cal J})$ has constant rank for $q\in B_{\epsilon}(p)$, which completes the proof.
\end{proof}

Clearly, the equivalent definition of RCRCQ given by Proposition~\ref{P:RCRCQ} is independent of the choice of the index set ${\cal K}$. It is easy to see that the definition of RCPLD is also independent of this choice. This concludes the analysis of strict implications depicted in Figure \ref{relations}, where in particular we have that RCPLD is strictly weaker than CPLD and RCRCQ.

At this point, condition RCPLD is the weakest one among the ones we have presented. Thus, we will prove that RCPLD is a strict CQ with respect to the AKKT condition, which will be true for all other conditions that imply it. Before doing this, let us present yet another CQ called {\it Constant Rank of the Subspace Component} (CRSC \cite{Andreani_Haeser_Schuverdt_Silva_CRSC_2012}). Noticing that while RCRCQ improves upon CRCQ by noticing that there is no reason to consider every subset of equality constraints, CRSC improves upon RCRCQ by noticing that the same thing is true with respect to the inequality constraints. Namely, it is not the case that every subset of the active inequality constraints must be taken into account; only a particular fixed subset of the constraints maintaining the constant rank property is enough for guaranteeing the existance of Lagrange multipliers. The definition is as follows:


\begin{definition}  \label{def:crsc}
Let $\Omega$  be  given by \eqref{eq:constset}, $p\in \Omega$,   ${\cal A}(p)$  and  ${\cal L} (p)^{\circ}$  be given by \eqref{eq:constset} and \eqref{eq:PolarConeLin}, respectively. Define the index set  ${\cal J}_{-}(p) = \left\{ j \in {\cal A}(p):~-\grad g_j (p)\in {\cal L} (p)^{\circ} \right\}$. The point  $p$ is said  to satisfy  CRSC if   there exists  $\epsilon >0$ such that  the rank of $\{\grad h_i(q):~i=1, \ldots , s \}\cup \{\grad g_j(q):~ j\in {\cal J}_{-}(p)\}$ is constant for all $q\in B_{\epsilon}(p)$. 
\end{definition}

It is clear that CRSC is weaker than RCRCQ and MFCQ, but its relation with RCPLD is not simple to establish. Somewhat surprisingly, CRSC is strictly weaker than RCPLD in the Euclidean setting. The proof is somewhat elaborate \cite{Andreani_Haeser_Schuverdt_Silva_CRSC_2012}, so we did not pursue it in the Riemannian setting, as it was not needed in our developments. Instead, we simply show that CRSC does not imply RCPLD in any Riemannian manifold ${\mathcal M}$ of dimension $n\geq2$. Thus, we shall prove the convergence of the augmented Lagrangian method under either of these conditions, even though we expect CRSC to be weaker than RCPLD. At this point, results under CRSC are at least independent of the ones where RCPLD are employed. We proceed with the example where CRSC holds and RCPLD fails:


\begin{example} \label{ex:CRSC_Not_RCPLD}
Let ${\cal M}$ be an $n$-dimensional smooth Riemannian manifold   with $n\geq  2$. Take  $p\in {\cal M}$ and   $g:=(g_1 , \ldots , g_{4})\colon{\cal M}\to {\mathbb R}^4$  continuously differentiable functions satisfying the following conditions:
\begin{enumerate}
\item[(i)]  $g(p)=0$;
\item[(ii)] $\grad g_1(p)= -\grad g_{2}(p)$ and $\grad g_3(p)= -\grad g_{4}(p)$;
\item[(iii)] for all $\epsilon >0$, there exists  $q \in B_{\epsilon} (p)$ such that the set  $\{\grad g_1({q}), \grad g_2({q})\}$ is linearly independent with  $q \neq p$;
\item[(iv)] the set $\{\grad g_1(p), \grad g_3(p)\}$ is linearly independent;
\item[(v)] there exists $\epsilon >0$ such that $ rank \{\grad g_j(q):~j=1, \ldots , 4\}=2$  for all $q \in B_{\epsilon} (p)$.

\end{enumerate}
Consider a feasible set $\Omega:=\{q\in {\cal M}:~  g(q)\leq 0\}$ and $p \in \Omega$. It follows from condition  $(ii)$, that the set $\{\grad g_1(p) , \grad g_2 (p)\}$ is positive linearly dependent. Hence, using conditions $(i)$ and $(iii)$, we conclude that RCPLD does not hold at $p$. In order to see that CRSC is valid at $p$, it is enough to note $(v)$ together with the fact that $(i)$ and $(ii)$ imply that ${\cal J}_{-}(p) = \left\{ 1, 2, 3, 4 \right\}$. Next, we will present two examples in which conditions $(i)$, $(ii)$, $(iii)$, $(iv)$, and $(v)$ are satisfied.
\begin{itemize} 
\item Consider  the sphere  ${\cal M}:=\{(x, y, z)\in {\mathbb R}^3:~x^2+y^2+z^2=1\}$ and take  $p:=(0,0,1)$, $g_1(x, y, z):=x-y^2$, $g_2(x, y, z):=-x$, $g_3(x, y, z):=y-x^2$ and $g_4(x, y, z):=-y$, where clearly $(i)$ holds. Similarly to Example \ref{ex:CPLDNotCRCQ_MFCQ}, Let $q:=(x,y,z)$ with $y\neq0$ and $z\neq0$, $u_1:=(1,-2y,0)$, $u_2:=(-1,0,0)$, $u_3:=(-2x,1,0)$, and $u_4:=(0,-1,0)$. Since $u_i\in\{p\}^{\perp}, i=1,2,3,4$, it is easy to see $(ii)$ and $(iv)$. In order to prove $(iii)$, notice that $\{q,u_1,u_2\}$ is linearly independent. Hence, since $\grad g_i(q)=\Pi_{\{q\}^{\perp}}(u_i)=u_i-r_{u_i}q, r_{u_i}\in\R$ for $i=1,2,3,4$, similarly to Example~\ref{ex:CPLDNotCRCQ_MFCQ} we have that $\{\grad g_1({q}), \grad g_2({q})\}$ is linearly independent. To see that $(v)$ holds, take $\epsilon>0$ such that $z\neq0$ for all $q=(x,y,z)\in B_\epsilon(p)$. Hence, $\{u_2, u_4, q\}$ is linearly independent, which implies that $\{\grad g_2({q}), \grad g_4({q})\}$ is linearly independent. From the fact that ${\cal M}$ is $2$-dimensional, $(v)$ holds.

\item Consider a $2$-dimensional complete manifold ${\mathcal M}$ and define the functions $g_1(q):=\varphi_1(q)$, $g_2(q):=\varphi_2(q)$,   $g_3(q):=\varphi_3(q)$, and $g_4(q):=-\varphi_3(q)$ from \eqref{defphi}.  Similarly to the computations in Example~\ref{ex:CPLDNotCRCQ_MFCQ}, one can prove that items  $(i)$, $(ii)$, $(iii)$, and  $(iv)$ are satisfied.  By $(iv)$ and Lemma~\ref{lemma:LD}, we have that the rank of $\{\grad g_j(q):~j=1, \ldots , 4\}$ is at least $2$ for all $q \in  B_{{\epsilon}}(p)$ and some $\epsilon>0$. Therefore, taking into account that ${\cal M}$ is $2$-dimensional, condition $(v)$ is also satisfied. Figure~\ref{fig:CRSC Not RCPLD}  illustrates this example.
\end{itemize}

\begin{figure}[!htb]
\begin{subfigure}[b]{0.45\linewidth}
\centering
\begin{tikzpicture}[scale=0.4, transform shape]
\tkzInit[xmin=-10,xmax=10,ymin=-6,ymax=6];
\tkzDefPoint(-7,-1){A}
\tkzDefPoint(7,1){B}
\tkzDefPoint(3,-5){C}
\fill [magenta!10!white] (-7,-1) to [out=50, in=140] (7,1) to [out=188, in=75] (3,-5) to [out=110, in=-7] (-7,-1);
\draw (-7,-1) to [out=50, in=140] (7,1) to [out=188, in=75] (3,-5) to [out=110, in=-7] (-7,-1);
\tkzDefPoint(-2.35,2.35){D}
\tkzDefPoint(-5.6,0){E}
\tkzDefPoint(4,-0.99){F}
\tkzDefPoint(6,1.6){G}
\tkzDrawSegment(D,E)
\tkzDrawSegment(E,F)
\tkzDrawSegment(F,G)
\tkzDrawSegment(G,D)
\tkzDrawPolygon[fill=cyan!10](D,E,F,G);
\tkzDefPoint(-0.06,0.73){O}
\tkzDefPoint(4,3.5){I}
\tkzDefPoint(4,0.5){J}
\tkzDefPoint(4.1,0.687){L}
\tkzDefPoint(-6,-0.7){M}
\tkzLabelPoint[left, magenta!100!red](M){${\cal M}$}
\tkzDefPoint(-4.6,-0.1){S}
\tkzLabelPoint[above, cyan!100!blue, font=\fontsize{8}{5}\selectfont](S){$T_p{\cal M}$}
\tkzDefPoint(1.6,0.1){Q1}
\tkzDefPoint(-1.95,0.7){Q2}
\tkzDefPoint(0.2,1.7){Q3}
\node[right, magenta1] at (Q1) {$p_2$};
\node[left, magenta2] at (Q2) {$p_1$};	
\node[right, magenta3] at (Q3) {$p_3$};	
\draw[thick,magenta2, opacity=0.3, smooth, line width=0.5pt] [rotate=40] plot[domain=-161:168,variable=\t] ({1.8*cos(\t)-1}, {sin(\t)*1.9+1.8},{sin(\t)});
\draw[thick,magenta2, opacity=1, smooth, line width=0.5pt] [rotate=40] plot[domain=168:254,variable=\t] ({1.8*cos(\t)-1}, {sin(\t)*1.9+1.8},{sin(\t)});
\draw[thick,magenta1, opacity=0.3, smooth,  line width=0.5pt] [rotate=-20] plot[domain=-21:228.5,variable=\t] ({1.8*cos(\t)+1.5},{ sin(\t)*1.9+0.7},{sin(\t)});
\draw[thick,magenta1, opacity=1, line width=0.5pt] [rotate=-20] plot[domain=228.5:343,variable=\t] ({1.8*cos(\t)+1.5},{ sin(\t)*1.9+0.7},{sin(\t)});
\draw[thick, magenta3, rotate=20, smooth, opacity=1, line width=0.5pt] plot[domain=-9.3:122.5, variable=\t] ({1.3*cos(\t)+0.9},{ 1.5*sin(\t)+1.7},{sin(\t)+0.4});	
\draw[thick, magenta3, rotate=20, smooth, opacity=0.3, line width=0.5pt] plot[domain=122.5:355.7, variable=\t] ({1.3*cos(\t)+0.9},{ 1.5*sin(\t)+1.7},{sin(\t)+0.4});			
\draw[magenta!50!white] plot [smooth] coordinates {(Q1) (O) (Q2)}; 
\draw[magenta!50!white] plot [smooth] coordinates {(Q3) (0.12,1.5) (O) }; 
\draw[my tip=stealth, magenta1] (-0.06,0.73) -- (1.8,0.55) node [above right, magenta1] {$\grad g_2 (p)$};
\draw[my tip=stealth, magenta1] (-0.06,0.73) -- ($(1.8,0.55)!2!(-0.06,0.73)$)node [above left, magenta1] {$\grad g_1 (p)$};
\draw[my tip=stealth, magenta3] (-0.06,0.73) -- (0.13,1.75) node[left , magenta3] {$\grad g_3 (p) $};
\draw[my tip=stealth, magenta3] (-0.06,0.73) -- ($(0.13,1.75)!2!(-0.06,0.73)$)node[right, magenta3] {$\grad g_4 (p) $};
\tkzLabelPoint[below left](O){$p$}
\tkzDrawPoint(O)
\node[circle, magenta1, fill, inner sep=0.8pt] (Q1) at (1.6,0.1){};
\node[circle, magenta2, fill, inner sep=0.8pt] (Q2) at (-1.95,0.7){};
\node[circle, magenta3, fill, inner sep=0.8pt] (Q3) at (0.2,1.7){};
\end{tikzpicture}
\caption{ The scenario at point $p$.}
\end{subfigure}
\hspace{0.5cm}
\begin{subfigure}[b]{0.45\linewidth}
\centering
\begin{tikzpicture}[scale=0.4, transform shape]
\tkzInit[xmin=-10,xmax=10,ymin=-6,ymax=6];
\tkzDefPoint(-7,-1){A}
\tkzDefPoint(7,1){B}
\tkzDefPoint(3,-5){C}
\fill [magenta!10!white] (-7,-1) to [out=50, in=140] (7,1) to [out=188, in=75] (3,-5) to [out=110, in=-7] (-7,-1);
\draw (-7,-1) to [out=50, in=140] (7,1) to [out=188, in=75] (3,-5) to [out=110, in=-7] (-7,-1);
\tkzDefPoint(-2.35,2.7){D}
\tkzDefPoint(-5.6,-0.55){E}
\tkzDefPoint(3.7,-1.55){F}
\tkzDefPoint(6,2.05){G}
\tkzDrawSegment(D,E)
\tkzDrawSegment(E,F)
\tkzDrawSegment(F,G)
\tkzDrawSegment(G,D)
\tkzDrawPolygon[fill=cyan!17](D,E,F,G);
\tkzDefPoint(-0.06,0.73){O}
\tkzDefPoint(-0.09,0.53){Q}
\tkzDefPoint(4,3.5){I}
\tkzDefPoint(4,0.5){J}
\tkzDefPoint(4.1,0.687){L}
\tkzDefPoint(-6,-0.7){M}
\tkzLabelPoint[left, magenta!100!red](M){${\cal M}$}
\tkzDefPoint(-4.6,-0.6){S}
\tkzLabelPoint[above, cyan!100!blue, font=\fontsize{8}{5}\selectfont](S){$T_q{\cal M}$}
\node[circle, magenta1, fill, inner sep=0.7pt] (Q1) at (1.6,0.1){};
\node[below, magenta1] at (Q1) {$p_2$};
\node[circle, magenta2, fill, inner sep=0.7pt] (Q2) at (-1.95,0.7){};
\node[below, magenta2] at (Q2) {$p_1$};	
\node[circle, magenta3, fill, inner sep=0.7pt] (Q3) at (0.25,1.75) {};
\node[below right, magenta3] at (Q3) {$p_3$};	
\draw[thick,magenta2, opacity=0.3, smooth, line width=0.3pt] [rotate=40] plot[domain=0:360,variable=\t] ({1.8*cos(\t)-1}, {sin(\t)*1.9+1.8},{sin(\t)});
\draw[thick,magenta4, opacity=0.3, smooth, dash pattern=on 0.2pt off 0.3pt, line width=0.3pt] [rotate=40] plot[domain=0:360,variable=\t] ({1.85*cos(\t)-1}, {sin(\t)*1.95+1.8},{sin(\t)});
\draw[thick,magenta1, opacity=0.3, smooth,  line width=0.3pt] [rotate=-20] plot[domain=-49.4:260.1,variable=\t] ({1.8*cos(\t)+1.5},{ sin(\t)*1.9+0.7},{sin(\t)});
\draw[thick,magenta1, opacity=1, line width=0.3pt] [rotate=-20] plot[domain=260.1:310.6,variable=\t] ({1.8*cos(\t)+1.5},{ sin(\t)*1.9+0.7},{sin(\t)});
\draw[thick,magenta4, opacity=0.3, smooth, dash pattern=on 0.2pt off 0.3pt,  line width=0.3pt] [rotate=-20] plot[domain=-46:257,variable=\t] ({1.85*cos(\t)+1.5},{ sin(\t)*1.95+0.7},{sin(\t)});
\draw[thick,magenta4, opacity=1, dash pattern=on 0.2pt off 0.3pt, line width=0.3pt] [rotate=-20] plot[domain=257:314,variable=\t] ({1.85*cos(\t)+1.5},{ sin(\t)*1.95+0.7},{sin(\t)});	
\draw[thick, magenta3, rotate=20, opacity=0.3, smooth, line width=0.3pt] plot[domain=96.1:379.5, variable=\t] ({1.3*cos(\t)+0.9},{ 1.5*sin(\t)+1.7},{sin(\t)+0.4});
\draw[thick, magenta3, rotate=20, opacity=1, smooth, line width=0.3pt] plot[domain=19.5:96.1, variable=\t] ({1.3*cos(\t)+0.9},{ 1.5*sin(\t)+1.7},{sin(\t)+0.4});	
\draw[thick, magenta3, rotate=20, opacity=0.3, smooth, dash pattern=on 0.2pt off 0.3pt, line width=0.3pt] plot[domain=108.5:370, variable=\t] ({1.5*cos(\t)+0.9},{ 1.7*sin(\t)+1.7},{sin(\t)+0.4});	
\draw[thick, magenta3, rotate=20, opacity=1, smooth, dash pattern=on 0.2pt off 0.3pt, line width=0.3pt] plot[domain=10:108.5, variable=\t] ({1.5*cos(\t)+0.9},{ 1.7*sin(\t)+1.7},{sin(\t)+0.4});	
\draw[my tip=stealth, magenta2] (-0.09,0.53) -- (1.8,0.8) node [right, magenta2] {$\grad g_2 (q)$};
\draw[my tip=stealth, magenta1] (-0.09,0.53) -- ($(1.8,0)!2!(-0.09,0.53)$) node [left, magenta1] {$\grad g_1 (q) $};
\draw[my tip=stealth, magenta3] (-0.09,0.53) --  (0.15,1.8) node[above left , magenta3] {$\grad g_3 (q) $};
\draw[my tip=stealth, magenta3] (-0.09,0.53) -- ($(0.15,1.8)!2!(-0.09,0.53)$)  node[left, magenta3] {$\grad g_4 (q) $};
\tkzLabelPoint[above left](O){$p$}
\tkzDrawPoint(O)
\tkzLabelPoint[below left](Q){${q}$}
\tkzDrawPoint(Q)

\end{tikzpicture}
\caption{Case in the neighborhood of the $p$ }
\end{subfigure}

	\caption{Illustrative figure for Example~\ref{ex:CRSC_Not_RCPLD} where RCPLD fails but CRSC is satisfied. Although the subset of gradients indexed by $\{1,2\}$ loses positive linear dependence for $q$ near $p$, all gradients remain to span a $2$-dimensional vector space (the whole tangent space) for $q$ near $p$.}
\label{fig:CRSC Not RCPLD}
\end{figure}
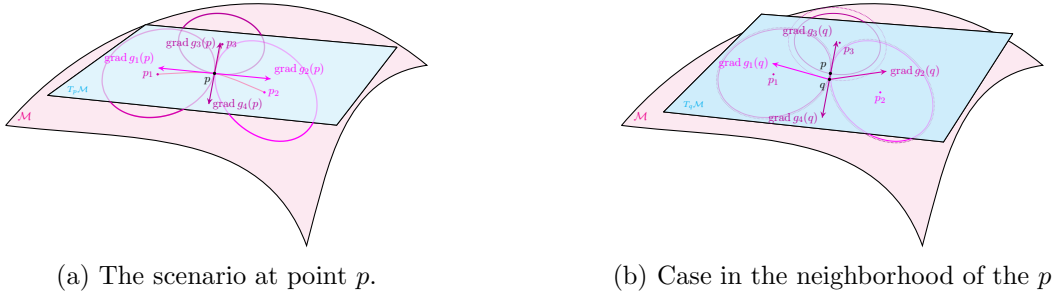

\end{example}

In our view, CRSC is the most interesting one of all previously defined conditions. Although we do not pursue its extensions to the Riemannian setting, we mention a few of its properties known in the Euclidean case. First, it has an elegant mathematical description. Also, the index set ${\cal J}_{-}(p)$ can be viewed as the index set of active inequality constraints that are treated as equality constraints in the polar of the linearized cone ${\cal{L}}^\circ(p)$ \eqref{eq:PolarConeLin}. Actually, this interpretation holds also for the cone ${\cal{L}}(p)$, since ${\cal J}_{-}(p)$ can be equivalently stated as the set of indexes $j\in{\cal A}(p)$ such that $\left\langle \grad g_j(p) , v \right\rangle = 0$ for all $v\in{\cal{L}}(p)$ \cite{positivity}. However, surprisingly, these interpretations are not by chance, since it was proved \cite{Andreani_Haeser_Schuverdt_Silva_CRSC_2012} that when $p\in\Omega$ satisfies CRSC, the constraints $g_j(q)\leq0, j\in{\cal J}_{-}(p)$ can only be satisfied as equalities for a feasible point $q$ in a small enough neighborhood around $p$. That is, one can safely replace the inequalities $g_j(q)\leq0, j\in{\cal J}_{-}(p)$ with equalities $g_j(q)=0, j\in{\cal J}_{-}(p)$ without locally altering the feasible set and in such a way that MFCQ holds. This result is connected to the one in \cite{shulu} which shows that whenever CRCQ is satisfied, there exists a local reformulation of the problem such that MFCQ holds. In fact, this procedure is well known in linear conic programming as {\it facial reduction}, that is, when the constraints are such that there is no interior point, there is an efficient procedure to replace the cone with one of its faces in such a way that a relative interior point exists. The extension of CRSC to the conic context and its connections with the facial reduction procedure are described in \cite{CRCQfaces2}, where they also show that CRSC may also provide the strong second-order necessary optimality condition depending on a single Lagrange multiplier by considering the constant rank property for all subsets of constraints that include ${\cal J}_{-}(p)$ and all equalities. Finally, we mention an additional property that holds in the Euclidean setting for all CQs discussed in this paper, that is, that they all imply that an error bound can be computed. We state this as the following conjecture in the Riemannian setting:

{\bf Conjecture:} Let ${\cal M}$ be a complete Riemannian manifold with dimension $n\geq 2$. Let $p\in\Omega$ be such that CRSC or RCPLD is satisfied. Then, there exists $\epsilon>0$ and $\alpha>0$ such that $$\inf_{w\in\Omega}d(q,w)\leq\alpha\max\{0,g_1(q),\dots,g_m(q),|h_1(q)|,\dots,|h_s(q)|\}$$ for all $q\in B_\epsilon(p)$.

Finally, let us prove that all conditions proposed so far are constraint qualifications. We do this by showing that they are strict CQs with respect to the necessary optimality condition AKKT from Theorem~\ref{def:AKKT}, since this gives us the main result of global convergence of Algorithm~\ref{Alg:LAA}. We start by showing that CRSC is a strict CQ, and we note that when the condition was introduced in the Euclidean setting \cite{Andreani_Haeser_Schuverdt_Silva_CRSC_2012}, only an indirect proof of this fact was presented. Thus, a clear direct proof was not available in the literature even in the Euclidean setting.

\begin{theorem} 
Suppose that $p\in\Omega$ satisfies CRSC. If $p$ is an AKKT point, then $p$ is a KKT point.
\end{theorem}
\begin{proof}
As in Definition~\ref{def:crsc}, let ${\cal J}_{-}(p) := \left\{ j \in {\cal A}(p):~-\grad g_j(p) \in {\cal L} (p)^{ \circ} \right\}$ and we denote ${\cal J}_+(p):={\cal A}(p)\backslash{\cal J}_-(p)$. Since $p$ is an AKKT point of problem~\eqref{PNL},  there exist   sequences $(p^k)_{k\in {\mathbb N}}\subset {\cal M}$, $(\lambda^k)_{k\in {\mathbb N}}\subset {\mathbb R}^s$, and $(\mu^k)_{k\in {\mathbb N}}\subset {\mathbb R}_+^m$ with $\mu^k_j=0$, for all $j\notin {\cal A}(p)$ such that    $\lim_{k\to \infty}p^k=p$ and 
\begin{equation}\label{neweq:akktcrsc2}
\grad f(p^k)+\sum_{i=1}^s\lambda_i^k\grad h_i(p^k)+ \sum_{\ell \in {\cal J}_{-}(p)} \mu_\ell^k \grad g_\ell(p^k)+\sum_{j \in {\cal J}_{+}(p)} \mu_j^k \grad g_j(p^k)=:\epsilon_k, \quad \forall k\in {\mathbb N}, 
\end{equation}
where $\lim_{k\to \infty} \epsilon_k=0$. Let ${\cal I}\subset\{1,\dots,s\}$ and ${\cal J}\subset{\cal J}_{-}(p)$ be such that $A(p,{\cal I},{\cal J})$ is a basis for the subspace generated by $A(p,\{1,\dots,s\},{\cal J}_-(p))$ (here we use the notation introduced in \eqref{eq:ap}). Thus, \eqref{neweq:akktcrsc2} can be rewritten as 
\begin{equation}\label{aux000}
\grad f(p^k)+\sum_{i\in{\cal I}}\tilde{\lambda}_i^k\grad h_i(p^k)+ \sum_{\ell \in {\cal J}} \tilde{\mu}_\ell^k \grad g_\ell(p^k)+\sum_{j \in {\cal J}_{+}(p)} \mu_j^k \grad g_j(p^k)=\epsilon_k, \quad \forall k\in {\mathbb N}, 
\end{equation}
for suitable $\tilde{\lambda}_i^k\in\R, i\in{\cal I}$ and $\tilde{\mu}_{\ell}^k\in\R, \ell\in{\cal J}$. If all the sequences $(\tilde{\lambda}_i^k)_{k\in\mathbb{N}}, i\in{\cal I}$, $(\tilde{\mu}_{\ell}^k)_{k\in\mathbb{N}}, \ell\in{\cal J}$, and $(\mu_j^k)_{k\in\mathbb{N}}, j\in{\cal J}_+$ are bounded, we may take a suitable convergent subsequence $(\tilde{\lambda}_i^k)_{k\in K_1}\to\tilde{\lambda}_i\in\R, i\in{\cal I}$, $(\tilde{\mu}_{\ell}^k)_{k\in K_1}\to\tilde{\mu}_{\ell}\in\R, \ell\in{\cal J}$, and $(\mu_j^k)_{k\in K_1}\to\mu_{j}\in\R_+, j\in{\cal J}_+$ such that
\begin{equation*}
\grad f(p)+\sum_{i\in{\cal I}}\tilde{\lambda}_i\grad h_i(p)+ \sum_{\ell \in {\cal J}} \tilde{\mu}_\ell \grad g_\ell(p)+\sum_{j \in {\cal J}_{+}(p)} \mu_j \grad g_j(p)=0.
\end{equation*}
Let us see that this implies that $p$ is a KKT point. First, note that ${\cal J}\subset{\cal J}_{-}(p)$ with ${\cal J}_{-}(p)\cup{\cal J}_{+}(p)=A(p)$. If some $\tilde{\mu}_{\ell_0}<0$, $\ell_0\in{\cal J}$, by the definition of the set ${\cal J}_{-}(p)$, we have $\tilde{\mu}_{\ell_0}\nabla g_{\ell_0}\in{\mathcal L}^{\circ}(p)$. But from \eqref{eq:PolarConeLin}, one can see that ${\mathcal L}^{\circ}(p)$ is closed under addition, which implies that $-\grad f(p)\in{\mathcal L}^{\circ}(p)$, that is, $p$ is a KKT point.
Otherwise, if it is not the case that all sequences are bounded, let us take a subsequence $K_2\subset\mathbb{N}$ such that $\lim_{k\in K_2}M_k=+\infty$, where $M_k=\max\{|\tilde{\lambda}^k_i|, i\in{\cal I}; |\tilde{\mu}^k_{\ell}|, \ell\in{\cal J}; \mu^k_j, j\in{\cal J}_+(p)\}$. Dividing \eqref{aux000} by $M_k$ and taking the limit on a suitable subsequence $K_3\subset K_2$ such that $\lim_{k\in K_3}\frac{\tilde{\lambda}^k_i}{M_k}=\alpha_i\in\R$, $\lim_{k\in K_3}\frac{\tilde{\mu}^k_{\ell}}{M_k}=\beta_\ell\in\R$, and $\lim_{k\in K_3}\frac{\mu^k_j}{M_k}=\gamma_j\geq0$ with not all $\alpha_i, \beta_\ell, \gamma_j$ equal to zero, we arrive at
$$
\sum_{i\in{\cal I}}\alpha_i\grad h_i(p)+ \sum_{\ell \in {\cal J}} \beta_\ell \grad g_\ell(p)+\sum_{j \in {\cal J}_{+}(p)} \gamma_j \grad g_j(p)=0.
$$
However, by the definition of ${\cal J}_+(p)$, we must have $\gamma_j=0$ for all $j\in{\cal J}_+(p)$, since otherwise, by replacing the scalars when some $\beta_{\ell}<0$ (as previously done) we would have $-\nabla g_j(p)\in{\cal L}(p)^{\circ}$. Hence $A(p,{\cal I},{\cal J})$ is linearly dependent. This contradicts the definition of the index sets ${\cal I}$ and ${\cal J}$.
\end{proof}

Similarly, we show that RCPLD is a strict CQ.

\begin{theorem} 
Suppose that $p\in\Omega$ satisfies RCPLD. If $p$ is an AKKT point, then $p$ is a KKT point.
\end{theorem}
\begin{proof} The proof is similar to the previous one, however without partitioning $A(p)$. That is, consider the previous proof with ${\cal J}_{-}(p)$ replaced by $\emptyset$ and ${\cal J}_{+}(p)$ replaced by ${\cal A}(p)$. We arrive similarly to \eqref{aux000} to a sequence
$$
\grad f(p^k)+\sum_{i\in{\cal I}}\tilde{\lambda}_i^k\grad h_i(p^k)+\sum_{j \in {\cal A}(p)} \mu_j^k \grad g_j(p^k)=\epsilon_k, \quad \forall k\in {\mathbb N}, 
$$
with $\tilde{\lambda}_i^k\in\R, i\in{\cal I}$, $\mu_j^k\geq0, j\in{\cal A}(p)$, $\lim_{k\in\mathbb{N}}p^k=p$, $\lim_{k\in\mathbb{N}}\epsilon_k=0$, and $\{\grad h_i(p): i\in{\cal I}\}$ linearly independent.

For every $k\in\mathbb{N}$, we apply Lemma \ref{l:Caratheodory} to arrive at
\begin{equation}\label{aux001}
\grad f(p^k)+\sum_{i\in{\cal I}}\bar{\lambda}_i^k\grad h_i(p^k)+\sum_{j \in {\cal J}_k} \bar{\mu}_j^k \grad g_j(p^k)=\epsilon_k,
\end{equation}
for some $\bar{\lambda}_i^k\in\R, i\in{\cal I}$, $\bar{\mu}_j^k\geq0, j\in{\cal J}_k\subset{\cal A}(p)$, and all $k\in {\mathbb N}$, where $A(p^k,{\cal I},{\cal J}_k)$ is linearly independent. Let us take a subsequence such that ${\cal J}_k$ is constant, say, ${\cal J}_k\equiv{\cal J}$ for all $k\in K_1\subset\mathbb{N}$. The proof now follows similarly to the previous one considering $M_k:=\max\{|\bar{\lambda}_i^k|, i\in{\cal I};\bar{\mu}_j^k, j\in{\cal J}\}$. If $(M_k)_{k\in K_1}$ is bounded, one may take the limit in \eqref{aux001} for a suitable subsequence to see that $p$ is a KKT point. Otherwise, dividing \eqref{aux001} by $M_k$ we see that $A(p,{\cal I},{\cal J})$ is positive-linearly dependent, which contradicts the definition of RCPLD.
\end{proof}

We formalize our results in the following:

\begin{corollary}
\label{cor}
Let $p$ be a feasible limit point of a sequence $(p^k)_{k\in\mathbb{N}}$ generated by Algorithm~\ref{Alg:LAA} such that $p$ satisfies RCPLD or CRSC. Then $p$ satisfies the KKT conditions.
\end{corollary}

Notice that differently from the result under MFCQ, where the dual AKKT sequence $(\lambda^k,\mu^k)_{k\in\mathbb{N}}$ is necessarily bounded, and thus dual convergence to a Lagrange multiplier is obtained, our result does not include convergence of the dual sequence. In the next section we prove that this can be obtained under a condition weaker than CPLD and independent of RCPLD known as quasinormality \cite{hestenes}. In order to do this, we shall extend a stronger sequential optimality condition to the Riemannian setting known as Positive-AKKT condition (PAKKT \cite{Andreani_Fazzio_Schuverdt_Secchin2019}).
   
\section{A stronger sequential optimality condition} \label{eq:PAKKT}

The quasinormality constraint qualification (QN) was introduced in \cite{hestenes} and it has been popularized in the book \cite{Bertsekas2016} in connection with convergence of the external penalty method. Recently, it has been connected with the notion of so-called Enhanced KKT conditions, guaranteeing boundedness of the corresponding set of Enhanced Lagrange multipliers \cite{enhanced}. QN is a fairly weak CQ, been known to be strictly weaker than CPLD \cite{Andreani_Martinez_Schuverdt2005} while still implying the Error Bound property \cite{eb2} in the Euclidean setting. In this section we will extend an important algorithmic property of QN that goes beyond what we have proved for RCPLD and CRSC. That is, besides QN being a strict CQ with respect to the AKKT condition, namely, global convergence of Algorithm~\ref{Alg:LAA} in the sense of Corollary \ref{cor} is also valid under QN, we will show that the dual sequence generated by Algorithm~\ref{Alg:LAA} under QN is in fact bounded. In order to do this, we will show that QN is a strict CQ with respect to a stronger sequential optimality condition known as Positive-AKKT (PAKKT) condition \cite{Andreani_Fazzio_Schuverdt_Secchin2019}.


We start by introducing the PAKKT condition in the Riemannian setting, showing that it is indeed a genuine necessary optimality condition for problem~\eqref{PNL}. Our definition considers a modification of the original one as suggested in \cite{Andreani_Haeser_Schuverdt_Secchin_Silva2022}.
\begin{definition}\label{def:pAKKT}
The Positive-Approximate-KKT (PAKKT) condition  for problem~\eqref{PNL} is satisfied at a  point $p\in \Omega$ if there exist sequences $(p^k)_{k\in {\mathbb N}}\subset {\cal M}$, $(\lambda^k)_{k\in {\mathbb N}}\subset {\mathbb R}^s$ and $(\mu^k)_{k\in {\mathbb N}}\subset {\mathbb R}_+^m$ such that
\begin{enumerate}
	\item[(i)] $\lim_{k\to \infty}\ p^k =  p$;
  \item[(ii)] $\lim_{k\to \infty} \left\| \grad L(p^k, \lambda^k, \mu^k)\right\| =  0$;
	\item[(iii)] $\mu_j^k=0$ for all $j\not\in{\cal A}(p)$ and sufficiently large $k$;
	\item[(iv)] If $\gamma_k := \left\|(1,\lambda^k,\mu^k)\right\|_{\infty}\to+\infty$ it holds:
\begin{equation}\label{df.PAKKTcontroleSinalRestIguald}
 \lim_{k\to \infty} \frac{\left|\lambda_i^k\right|}{\gamma_k} >  0 \quad   \Longrightarrow  \quad   \lambda_i^kh_i(p^k) >0,  \, \forall k\in {\mathbb N}; 
\end{equation}
\begin{equation}\label{df.PAKKTcontroleSinalRestDesiguald}
 \lim_{k\to \infty} \frac{\mu_j^k}{\gamma_k} >  0  \quad   \Longrightarrow  \quad    \mu_j^kg_j(p^k) >0,  \, \forall k\in {\mathbb N}.
\end{equation}
\end{enumerate}
\end{definition}
A point $p$ satisfying Definition~\ref{def:pAKKT} is called a PAKKT point; the correspondent sequence $(p^k)_{k\in\mathbb{N}}$ is its associated primal sequence while $(\lambda^k,\mu^k)_{k\in\mathbb{N}}$ is its associated dual sequence. 
In order to present our results, we will make use of the following lemmas extended to the Riemannian setting in \cite{Yamakawa_Sato2022}:
\begin{lemma} \label{le:leAux}
Let $p$ be a local minimizer of problem~\eqref{PNL} and $\alpha>0$. Then,  for each  $k\in {\mathbb N}$ and ${\rho_k}>0$,  the following problem  
\begin{equation*}
\begin{array}{l}   
\displaystyle\Min_{q\in {\cal M}}  f(q) +  \frac{1}{2} d(q,p)^2 + \frac{\rho_k}{2} \left(\left\|h(q)\right\|_2^2 + \left\|g(q)_{+}\right\|_2^2\right),\\
\mbox{subject~to~} d(q,p)\leq \alpha, 
\end{array}
\end{equation*}
admits a solution $p^k$. Moreover, if $\lim_{k\to \infty}\ {\rho_k}=+\infty$ then $\lim_{k\to \infty}\ p^k =  p$.
\end{lemma}
\begin{lemma} \label{le:leAuxs}
Let $\phi\colon {\cal M} \to {\mathbb R}$ be a differentiable function, $\alpha>0$ and $p_0\in {\cal M}$. Suppose that $p\in  {\cal M} $ is an optimal solution of the following optimization problem:
\begin{equation*}
\begin{array}{l}
\displaystyle\Min_{q\in {\cal M}}  \phi (q), \\
\mbox{subject~to~} d(q,p_0)\leq \alpha.
\end{array}
\end{equation*}
If $d(p,p_0)< \alpha$, then $\grad \phi (p) =0$.
\end{lemma}
We now show that PAKKT is a genuine necessary optimality condition for problem~\eqref{PNL}.
\begin{theorem}\label{T:PAKKTCondNec}
Let $p\in\Omega$ be a local minimizer of \eqref{PNL}. Then, $p$ is a PAKKT point. 
\end{theorem}
\begin{proof}
Let $p$ be a local minimizer of problem~\eqref{PNL}. Thus, there is a sufficiently small parameter $\alpha >0$ such that the problem
\begin{equation*}
\begin{array}{l}
\displaystyle\Min_{q\in {\cal M}}f(q)+ \frac{1}{2} d(q,p)^2,\\
\mbox{subject~to~}h(q)=0, ~ g(q)\leq 0,~ d(q,p)\leq \alpha, 
\end{array}
\end{equation*}
has  $p$ as the unique global minimizer. For each $k\in {\mathbb N}$, take ${\rho_k}>0$ such that $\lim_{k\to \infty}\ {\rho_k}=+\infty$. Consider the penalized problem
\begin{equation}\label{PNLaux2}
\begin{array}{l}
\displaystyle\Min_{q\in {\cal M}}  f(q) +  \frac{1}{2} d(q,p)^2 + \frac{\rho_k}{2} \left(\left\|h(q)\right\|_2^2 + \left\|g(q)_{+}\right\|_2^2\right),\\
\mbox{subject~to~} d(q,p)\leq \alpha.
\end{array}
\end{equation}
  It follows from Lemma~\ref{le:leAux}  that  there exists a sequence $(p^k)_{k\in {\mathbb N}}$ such that $p^k$ is a solution of \eqref{PNLaux2} and  $\lim_{k \rightarrow \infty} p^k = p$. Thus, item~$(i)$ of  Definition~\ref{def:pAKKT} is satisfied. Moreover,  there exists  an infinite index set ${K}_1$ such that $d(p^k,p)< \alpha$, for all $k\in {K}_1$. Consequently,  using Lemma~\ref{le:leAuxs},  we conclude that 
\begin{equation*}
\grad f(p^k) -\exp^{-1}_{p^k}{p}+\sum_{i=1}^s\rho_k h_i(p^k)\grad h_i(p^k)+ \sum_{j=1}^m\rho_k \max\{ 0, g_j(p^k)\}\grad g_j(p^k)  = 0,
\end{equation*}
for all $k\in {K}_1$.  Therefore, we have
\begin{align*}
\lim_{k \in {K}_1} \grad L(p^k, \lambda^k , \mu^k) &= \lim_{k \in {K}_1} \Big(\grad f(p^k) +\sum_{i=1}^s\lambda_i^k\grad h_i(p^k)+ \sum_{j=1}^m\mu_j^k\grad g_j(p^k)\Big)\\
                                                                                             &= \lim_{k \in {K}_1} -\exp^{-1}_{p^k}{p} = 0,
\end{align*}
where for each $k \in {K}_1$, we denote $\lambda^k: = \rho_k h(p^k)$ and $\mu^k := \rho_k [g(p^k)]_+ \geq 0$. Therefore, $(ii)$ and $(iii)$  of Definition~\ref{def:pAKKT}   are satisfied.  We will now analyze the validity of~\eqref{df.PAKKTcontroleSinalRestIguald}   and~\eqref{df.PAKKTcontroleSinalRestDesiguald}.   Let $\gamma_k := \left\|(1,\lambda^k,\mu^k)\right\|_{\infty}$ for all $k\in {K}_1$ be such that $\lim_{k\in K_1}M_k=+\infty$, and assume that  $\lim_{k\in {K}_1} ({\left|\lambda_i^k\right|}/{\gamma_k}) >  0$. Thus, ${\left|\lambda_i^k\right|}/{\gamma_k} >  0$ for sufficiently large $k\in {K}_1$, which implies that $h(p^k)\neq0$. Hence, $\lambda_i^kh_i(p^k)=\rho_kh_i(p^k)^2>0$ for all sufficiently large $k\in K_1$.  
Similarly, if $\lim_{k\in {K}_1} ({\mu_j^k}/{\gamma_k}) >  0$, we have $\mu_j^k>0$, which implies $g_j(p^k)>0$ for sufficiently large $k\in K_1$. 
Therefore, \eqref{df.PAKKTcontroleSinalRestIguald}   and \eqref{df.PAKKTcontroleSinalRestDesiguald} are fulfilled.  Consequently,   $p$ satisfies Definition~\ref{def:pAKKT}, which concludes the proof. 
\end{proof}
Let us now introduce QN in the Riemannian context. We will show that under QN the dual sequence $(\lambda^k,\mu^k)_{k\in\mathbb{N}}$ associated with any PAKKT sequence $(p^k)_{k\in\mathbb{N}}$ is bounded. Later, we will show that Algorithm~\ref{Alg:LAA} generates PAKKT sequences, which will provide the main algorithmic relevance of QN.
\begin{definition}\label{def:QN}
Let $\Omega$  be  given by \eqref{eq:constset}, $p\in \Omega$ and  ${\cal A}(p)$ be given by~\eqref{eq:actset}. The point $p$ satisfies the  quasinormality constraint qualification (QN) if there are no $\lambda \in {\mathbb R}^s$ and $\mu \in {\mathbb R}_{+}^{m}$ such that
\begin{enumerate}
	\item[(i)] $\sum_{i=1}^{s} \lambda_i \grad h_i (p) + \sum_{j \in {\cal A}(p)} \mu_j \grad g_j(p)=0$; 
	\item [(ii)] $\mu_j=0$ for all $j \notin {\cal A}(p)$ and $(\lambda, \mu) \neq 0$;
	\item [(iii)] for all $\epsilon >0$ there exists $q \in B_{\epsilon}(p)$ such that $\lambda_i h_i(q)>0$ for all $i \in\left\{1, \ldots , s \right\}$ with $\lambda_i \neq 0$ and $\mu_j g_j(q)>0$ for all $j \in {\cal A}(p)$ with $\mu_j>0$.
\end{enumerate}
\end{definition}

In the next example we show that QN holds, but both RCPLD and CRSC fail. 
\begin{example}\label{ex:QN_Not_RCPLD}
Define the functions $h_1(q):=\varphi_1(q)e^{\varphi_2(q)}$ and $h_2(q):=\varphi_1(q)$ as defined in \eqref{defphi}. Note that 
\begin{equation} \label{eq:igqn}
 \grad h_1(q):=e^{\varphi_2(q)} \grad \varphi_1(q)+ \varphi_1(q)e^{\varphi_2(q)} \grad \varphi_2(q), \qquad \grad h_2(q):=\grad \varphi_1(q).
\end{equation}
The point   $p \in \Omega$ satisfies QN. Indeed,  first we note that  $\varphi_1(p)=0$ and $\varphi_2(p)=0$.  Moreover,  we have $\grad h_1(p)=\grad h_2(p)=\grad \varphi_1(p)$.  Consider  the linear combination $\lambda_1 \grad h_1(p)+ \lambda_2\grad h_2(p)=0$ with $\lambda_1$ and $\lambda_2 \in {\mathbb R}$. Thus, we have 
$
(\lambda_1 + \lambda_2)\grad \varphi_1(p)=0.
$
Since $\grad \varphi_1(p)\neq 0$, we conclude that unless $\lambda_1=\lambda_2=0$, we must have $\lambda_1 \lambda_2<0$. In this case, take $\epsilon >0$ and  $q \in B_{\epsilon}(p)$ such that $q\neq p$.  Since  $h_1(q) h_2(q)>0$, we  conclude that  $\lambda_1 h_1(q)$ and $\lambda_2 h_2(q)$ have opposite signs, which implies that $p$ satisfies QN.

Now, we  are going to show that $p$ does not satisfy RCPLD nor CRSC.  For that, we first  note that  rank of $\{\grad h_1(p), \grad h_2(p) \}$ is equal to one.
On the other hand, similarly to the computations in Example~\ref{ex:CPLDNotCRCQ_MFCQ}, one can prove that, for all $\epsilon >0$, there exists  $q \in B_{\epsilon} (p)$ such that $\{\grad \varphi_1({q}), \grad \varphi_2({q})\}$ is linearly independent with  $q \neq p$.  By the definition of $\varphi_1$, notice that $\varphi_1(q)\neq 0$ for all $q \in B_{\epsilon} (p)$ and sufficiently small $\varepsilon>0$; it follows from \eqref{eq:igqn} that $\{\grad h_1({q}), \grad h_2({q})\}$ is  also linearly independent. Therefore, $p$ does not satisfy RCPLD nor CRSC.
\end{example} 
\begin{theorem}\label{T:PAKKT_QN+Bounded}
Let $p\in\Omega$ be a PAKKT point with associated primal sequence $(p^k)_{k\in {\mathbb N}}$ and dual sequence $(\lambda^k,\mu^k)_{k\in\mathbb{N}}$. Assume that $p$ satisfies QN. Then $(\lambda^k,\mu^k)_{k\in {\mathbb N}}$ is a bounded sequence. In particular, $p$ satisfies the KKT conditions and any limit point of $(\lambda^k,\mu^k)_{k\in\mathbb{N}}$ is a Lagrange multiplier associated with $p$.
\end{theorem}
\begin{proof}
Let $p\in\Omega$ be a PAKKT point with primal sequence $(p^k)_{k\in\mathbb{N}}$ and dual sequence $(\lambda^k,\mu^k)_{k\in\mathbb{N}}$ and let us assume that the dual sequence is unbounded. Then, we will conclude that the point $p$  does not satisfy the quasinormality condition, i.e., we will prove the existence of $\lambda \in {\mathbb R}^s$ and $\mu \in {\mathbb R}_{+}^{m}$ such that items~$(i)$, $(ii)$ and $(iii)$ of  Definition~\ref{def:QN}  are satisfied. For that, set $\gamma_k = \left\|(1,\lambda^k,\mu^k)\right\|_{\infty}$ as in Definition~\ref{def:pAKKT} and take an infinite subsequence indexed by $K_1$ such that $\lim_{k\in K_1}\gamma_k=+\infty$. To simplify the notations let  us  define the following auxiliary sequence 
\begin{equation*}
U^k:=(1,\lambda^k,\mu^k) \in {\mathbb R} \times {\mathbb R}^s \times {\mathbb R}_+^m, \qquad \forall k\in {\mathbb N},
\end{equation*}
with  $\lim_{k \in K_1} \|U^k\|_2 = \infty$. Take an infinite subset $K_2\subset K_1$ such that the sequence $(U^k/\|U^k\|_2)_{k\in {K}_2}$ converges to some $(0,\lambda,\mu)\in\R\times\R^s\times\R^m_+$, with $\|(0,\lambda,\mu)\|=1$. 
Thus, considering that $(p^k)_{k\in {\mathbb N}}$  is a primal PAKKT sequence,  we conclude that 
\begin{equation*}
 \lim_{k \in {K}_2}\frac{ \grad L(p^k, \lambda^k , \mu^k)}{\gamma_k} = \lim_{k \in {K}_2} \Big(\frac{\grad f(p^k)}{\gamma_k} +\sum_{i=1}^s\frac{\lambda_i^k}{\gamma_k}\grad h_i(p^k)+ \sum_{j=1}^m\frac{\mu_j^k}{\gamma_k}\grad g_j(p^k)\Big)= 0.
\end{equation*}
Hence, taking into account that $\mu_j=0$ for $j \notin {\cal A}(p)$, we obtain that item~$(i)$ of  Definition~\ref{def:QN} is satisfied at $p$. In addition , since $(\lambda, \mu)\neq 0$, item~$(ii)$ of Definition~\ref{def:QN} is also satisfied  at $p$. 
From \eqref{df.PAKKTcontroleSinalRestIguald} and \eqref{df.PAKKTcontroleSinalRestDesiguald}, we have that $\lambda_i^kh_i(p^k)>0$ whenever $\lambda_i\neq0$, and $\mu_j^kg_j(p^k)>0$ whenever $\mu_j>0$, which gives precisely item~$(iii)$ of Definition~\ref{def:QN}. Therefore, QN fails.
\end{proof}





Finally, it remains to show that Algorithm~\ref{Alg:LAA} generates PAKKT sequences, which gives its global convergence result under QN.

\begin{theorem}
Assume Algorithm~\ref{Alg:LAA} generates an infinite sequence $(p^k)_{k\in {\mathbb N}}$ with a feasible accumulation $p$, say, $\lim_{k\in K}p^k=p$. Then, $p$ is a PAKKT point with correspondent primal sequence $(p^k)_{k\in K}$ and dual sequence $(\lambda^k,\mu^k)_{k\in K}$ as generated by Algorithm~\ref{Alg:LAA}. In particular, $p$ is a KKT point and any limit point of $(\lambda^k,\mu^k)_{k\in K}$ is a Lagrange multiplier associated with $p$.
\end{theorem}
\begin{proof} 
By {\bf Step 1} and {\bf Step 2} of the algorithm, we have $$\lim_{k\in K}\grad L(p^k,\lambda^k,\mu^k)=\lim_{k\in K}\grad{\cal L}_{\rho_k}(p^k,\bar{\lambda}^k,\bar{\mu}^k)=0,$$
with $\mu_j^k=0$ for sufficiently large $k\in K$ if $j\not\in{\cal A}(p)$. To see this, note that when $(\rho_k)_{k\in K}$ is unbounded, this follows from the definition of $\mu_j^k$, the boundedness of $(\bar{\mu}_j^k)_{k\in K}$, and the fact that $\rho_kg_j(x^k)\to-\infty$. When $(\rho_k)_{k\in K}$ is bounded, we must have from {\bf Step 3} that $V^k\to0$. In particular, $\max\left\{0,\frac{\bar{\mu}^k_j}{\rho_k}+g_j(p^k)\right\}-\frac{\bar{\mu}_j^k}{\rho_k}\to0$. Since $g_j(p^k)<0$ is bounded away from zero when $j\not\in{\cal A}(p)$, we must have $\mu_j^k=0$ for sufficiently large $k\in K$. Thus $(i)$, $(ii)$, and $(iii)$ of Definition~\ref{def:pAKKT} hold.

To prove $(iv)$ of Definition \ref{def:pAKKT}, let $\gamma_k := \left\|(1,\lambda^k,\mu^k)\right\|_{\infty}$ and assume that $\lim_{k\in K}\gamma_k=+\infty$. Let $i\in\{1,\dots,s\}$ be such that $\lim_{k\in K}\frac{\lambda_i^k}{\gamma_k}=\lambda_i\neq0$ and $j\in{\cal A}(p)$ be such that $\lim_{k\in K}\frac{\mu_j^k}{\gamma_k}=\mu_j>0$. This implies that $\lambda_i^k=\bar{\lambda}_i^k+\rho_kh_i(p^k)$ is unbounded for $k\in K$, with $\lambda_i^k\lambda_i>0$. Since $(\bar{\lambda}_i^k)_{k\in K}$ is bounded, the only possibility is that $\rho_k\to+\infty$ and $\lambda_ih_i(p^k)>0$ for sufficiently large $k\in K$. Similarly, we have $\mu_jg_j(p^k)>0$ for sufficiently large $k\in K$, which completes the proof.
\end{proof}

We conclude by providing another property of QN in connection with Algorithm~\ref{Alg:LAA}. Instead of considering \eqref{step1 alg1} in {\bf Step 1} of Algorithm~\ref{Alg:LAA}, one may consider a more flexible criterion for solving the correspondent subproblem. That is, instead of requiring the iterate $p^k$ to satisfy $\left\| \grad {\mathcal L}_{\rho_k}(p^k, \bar{\lambda}^k, \bar{\mu}^k)\right\| \leq\epsilon_k$, one may require the looser criterion $\left\| \frac{\grad {\mathcal L}_{\rho_k}(p^k, \bar{\lambda}^k, \bar{\mu}^k)}{\gamma_k}\right\| \leq\epsilon_k$, where $\gamma_k := \left\|(1,{\lambda}^k,{\mu}^k)\right\|_{\infty}$ with $\lambda^k$ and $\mu^k$ given as in {\bf Step 2} of the algorithm. That is, one abdicates robustness of the solution of the subproblem in place of an easier computable iterate. For instance, this is the approach considered in the well known interior point method IPOPT \cite{ipopt}, even though it tends to generate unbounded dual sequences \cite{ye}. This modification gives rise to the so-called Scaled-PAKKT condition \cite{Andreani_Haeser_Schuverdt_Secchin_Silva2022} which we present in the Riemannian setting as follows:


\begin{definition}\label{def:Scaled_PAKKT}
The  Scaled-PAKKT condition  for problem~\eqref{PNL} is satisfied at a  point $p\in \Omega$ if there exist   sequences $(p^k)_{k\in {\mathbb N}}\subset {\cal M}$, $(\lambda^k)_{k\in {\mathbb N}}\subset {\mathbb R}^s$ and $(\mu^k)_{k\in {\mathbb N}}\subset {\mathbb R}_+^m$  such that  it holds: 
\begin{enumerate}
	\item[(i)] $\lim_{k\to \infty}\ p^k =  p$;
  \item[(ii)] $\lim_{k\to \infty} \left\| \frac{\grad L(p^k, \lambda^k, \mu^k)}{\gamma_k}\right\| =  0$, where  $\gamma_k := \left\|(1,\lambda^k,\mu^k)\right\|_{\infty}$;
	\item[(iii)] $\mu_j^k=0$ for all $j\not\in{\cal A}(p)$ and sufficiently large $k$;
	\item[(iv)] If  $\gamma_k\to+\infty$, then 
\begin{equation*} 
 \lim_{k\to \infty} \frac{\left|\lambda_i^k\right|}{\gamma_k} >  0 \quad   \Longrightarrow  \quad   \lambda_i^kh_i(p^k) >0,  \, \forall k\in {\mathbb N}; 
\end{equation*}
\begin{equation*}
 \lim_{k\to \infty} \frac{\mu_j^k}{\gamma_k} >  0  \quad   \Longrightarrow  \quad    \mu_j^kg_j(p^k) >0,  \, \forall k\in {\mathbb N}.
\end{equation*}
\end{enumerate}
\end{definition}

It is easy to see that Algorithm~\ref{Alg:LAA} with the looser criterion in {\bf Step 1} as described previously generates Scaled-PAKKT sequences. Now, it is easy to see that QN is still sufficient for guaranteeing boundedness of the dual Scaled-PAKKT sequence, following the proof of Theorem \ref{T:PAKKT_QN+Bounded}; however, let us show that QN is somewhat the weakest condition with this property.

\begin{theorem}
If for each continuously differentiable function $f\colon {\cal M} \rightarrow {\mathbb R}$ such that $p\in \Omega$ is a Scaled-PAKKT point, the KKT conditions also hold, then $p$ satisfies QN or ${\cal L} (p)^{\circ } = T_p{\cal M}$.
\end{theorem}
\begin{proof}
Assume that the point $p$ does not satisfy  the  quasinormality condition and ${\cal L} (p)^{\circ} \neq T_p{\cal M}$. We will show that there exists a continuously differentiable function $f\colon {\cal M} \rightarrow {\mathbb R}$ such that  $p$ is a Scaled-PAKKT point, but $p$ is not a KKT point. Since ${\cal L} (p)^{\circ} \neq T_p{\cal M}$, taking into account that ${\cal L} (p)^{\circ} \subset T_p{\cal M}$ and $0 \in {\cal L} (p)^{\circ}$,  we concluded that there exists $v \in T_p{\cal M}$ with $v\neq 0$ such that $v \notin {\cal L} (p)^{\circ}$. Thus,  by  the definition of ${\cal L} (p)^{\circ}$ in \eqref{eq:PolarConeLin},  we have 
\begin{equation}\label{eq1:T:ScaledPAKKT} 
-v +\sum_{i=1}^{s} \bar{\lambda}_i \grad h_i (p) + \sum_{j \in {\cal A}(p)} \bar{\mu}_j \grad g_j(p) \neq 0, \quad \forall \bar{\mu}_j \geq 0, \, \forall \bar{\lambda}_i \in {\mathbb R}.
\end{equation}
To proceed we  take  $0<\delta<r_{p}$, the injectivity radius, and  define  $f\colon B_{\delta}({p}) \rightarrow {\mathbb R}$ by $f(q):=\langle -v, - \exp^{-1}_{p}{q}\rangle$, which is continuously  differentiable and    $\grad f(p)=-v$.  Hence, \eqref{eq1:T:ScaledPAKKT}  implies   that $p$  is not a KKT point. It remains to show that  $p$ is a Scaled-PAKKT point. Since $p$ does not satisfy QN, there exist $\lambda \in {\mathbb R}^s$ and $\mu \in {\mathbb R}_{+}^{m}$  that satisfy items~$(i)$, $(ii)$, and $(iii)$ of Definition~\ref{def:QN}. 
In particular, by item $(iii)$, let $(p^k)_{k \in {\mathbb N}} \subset {\cal M}$ be such that, $\lim_{k \rightarrow \infty} p^k = p$, with
\begin{equation}\label{eq2:T:ScaledPAKKT} 
\lambda_i h_i(p^k)>0, \quad \forall i \in\left\{1, \ldots , s \right\}, \,\, \lambda_i \neq 0 \quad and \quad \mu_j g_j(p^k)>0, \quad \forall j \in {\cal A}(p), \,\, \mu_j>0.
\end{equation}
By the  continuity of $\grad h$ and $\grad g$ and item~$(i)$ of Definiton~\ref{def:QN}, we have
\begin{equation}\label{eq3:T:ScaledPAKKT} 
\lim_{k \rightarrow \infty} \Big(\sum_{i=1}^{s} \lambda_i \grad h_i (p^k) + \sum_{j \in {\cal A}(p)} \mu_j \grad g_j(p^k)\Big)=0.
\end{equation}
It follows from~\eqref{eq3:T:ScaledPAKKT} and from the continuity of $\grad f$, provided that $\lim_{k \rightarrow \infty} p^k = p$, that 
\begin{equation}\label{eq4:T:ScaledPAKKT} 
\lim_{k \rightarrow \infty} \frac{1}{k}\Big(\grad f (p^k) + \sum_{i=1}^{s} k\lambda_i \grad h_i (p^k) + \sum_{j \in {\cal A}(p)} k\mu_j \grad g_j(p^k)\Big)=0.
\end{equation}
Taking into account that any positive multiple of $(\lambda, \mu)$ also satisfies the three items of Definition~\ref{def:QN}, we can suppose without loss of generality that $\left\|(\lambda, \mu)\right\|_\infty =1$. Thus, setting    $\lambda^k:= k \lambda$, $\mu^k := k \mu$, and  $\gamma_k := \left\|(1,\lambda^k,\mu^k)\right\|_{\infty}$  we have $\gamma_k = k$. Hence, \eqref{eq4:T:ScaledPAKKT}  becomes
\begin{equation}\label{eq5:T:ScaledPAKKT2} 
\lim_{k \rightarrow \infty} \frac{1}{\gamma_k}\Big(\grad f (p^k) + \sum_{i=1}^{s} \lambda^k_i \grad h_i (p^k) + \sum_{j \in {\cal A}(p)} \mu^k_j \grad g_j(p^k)\Big)=0.
\end{equation}
We conclude that $p$ is a scaled PAKKT point and the proof is complete.
\end{proof}

\section{Conclusions}

In this paper we presented a detailed global convergence analysis of a safeguarded augmented Lagrangian method defined on a complete Riemannian manifold. In order to do this, we presented several weak constraint qualifications that can be used to obtain stationarity of all limit points of a primal sequence generated by the algorithm, despite the fact that the dual sequence may be unbounded. In doing so, we described several properties of these conditions well known in the Euclidean setting, which should foster further developments in the Riemannian setting. By means of a stronger sequential optimality condition, we were able to present a weak constraint qualification which guarantees boundedness of the dual sequence, even when the true set of Lagrange multipliers is unbounded. In presenting our conditions, we provided illustrative examples to prove that our conditions are strictly weaker than previously known ones in {\it any} complete Riemannian manifold with dimension $n\geq2$.

Note that when defining the sequential optimality conditions, we chose to present the simplest complementarity measure, namely, item ii) of Theorem \ref{def:AKKT} and item iii) of Definition \ref{def:pAKKT}, while in \cite{Yamakawa_Sato2022}, they considered a slightly stronger complementarity measure known as Approximate Gradient Projection. See the recent discussion about several different complementarity measures in \cite{mor} in the context of Euclidean conic optimization. We foresee significant progress in this topic in the nearby future, in particular, several stronger first- and second-order global convergence results of augmented Lagrangian methods and other algorithms should be expected to be extended to the Riemannian setting.
 
Finally, in the particular case where the manifold ${\cal M}$ can be embedded in an Euclidean space, one can treat $x\in{\cal M}$ as a subproblem/lower level constraint as described in \cite{lowerlevel}. It is clear that one should exploit the Riemannian structure in order to solve the subproblems more efficiently, while it is also clear that an intrinsic formulation of the theory is well justified \cite{BergmannHerzog2019}; however, in this context, it is not clear whether the pure Euclidean theory differs from the one formulated in the Riemannian setting. This topic will be the subject of a forthcoming paper.


\bibliographystyle{habbrv}
\bibliography{SeqOtimCondRiemannian}

\end{document}